\documentclass[12pt]{article}

\usepackage{amssymb}
\usepackage{amsthm}
\usepackage{fleqn}
\usepackage{graphicx}
\usepackage{epstopdf}
\usepackage{amsfonts}
\usepackage[numbers]{natbib}
\usepackage[colorlinks,citecolor=blue,urlcolor=blue]{hyperref}
\usepackage{amsmath}
\usepackage{geometry}
\usepackage{pifont}
\newtheorem{theorem}{Theorem}[section]
\newtheorem{lemma}[theorem]{Lemma}

\newtheorem{remark}[theorem]{Remark}

\newtheorem{example}[theorem]{Example}

\begin{document}

\title{Linear quadratic leader-follower stochastic differential games
for mean-field switching diffusions\thanks{This work was supported by
the National Natural Science Foundation of China (11801072, 61873325, 11831010, 11771079),
the Fundamental Research Funds for the Central Universities (2242021R41175, 2242021R41082),
and a start-up fund at the Southern University of Science and Technology (Y01286120).}}

\author{
Siyu Lv\thanks{School of Mathematics, Southeast University,
Nanjing 211189, China (lvsiyu@seu.edu.cn).}
\and
Jie Xiong\thanks{Department of Mathematics
and National Center for Applied Mathematics (Shenzhen),
Southern University of Science and Technology,
Shenzhen 518055, China (xiongj@sustech.edu.cn).}
\and
Xin Zhang\thanks{School of Mathematics, Southeast University,
Nanjing 211189, China (x.zhang.seu@gmail.com).}
}

\date{}

\maketitle

\begin{abstract}
In this paper, we consider a linear quadratic (LQ) leader-follower stochastic differential game
for \emph{regime switching} diffusions with \emph{mean-field} interactions. One of the salient
features of this paper is that \emph{conditional} mean-field terms are included in the state equation
and cost functionals. Based on stochastic maximum principles (SMPs), the follower's problem
and the leader's problem are solved sequentially and an \emph{open-loop} Stackelberg equilibrium
is obtained. Further, with the help of the so-called four-step scheme, the corresponding
Hamiltonian systems for the two players are decoupled and then the open-loop Stackelberg equilibrium
admits a \emph{state feedback representation} if some \emph{new-type} Riccati equations are solvable.
\end{abstract}


\section{Introduction}

The leader-follower game involves two players with \emph{asymmetric} roles,
one called the leader and the other called the follower. In the game,
the leader first announces her action, and the follower, according to the
leader's action, chooses an optimal response to minimize his cost functional.
Next, the leader has to take the follower's optimal response into account and
chooses an optimal action to minimize her cost functional. Yong \cite{Yong2002}
first considered a linear quadratic (LQ) leader-follower stochastic differential
game. Then, within the LQ framework, the result was extended by, e.g.,
\cite{SWX2016,Moon2021,LXY2021} in different settings.

Mean-field stochastic differential equations (SDEs) were initially suggested to
describe physical systems involving a \emph{large} number of \emph{interacting}
particles. In the dynamics of a mean-field SDE, one replaces the interactions
of all the particles by their \emph{average} or \emph{mean} to reduce the complexity
of the problem. In the last decade, since Buchdahn et al. \cite{BDLP2009,BLP2009}
and Carmona and Delarue \cite{CD2013-1,CD2013-2,CD2015} introduced the mean-field
backward SDEs (BSDEs) and mean-field forward-backward SDEs (FBSDEs), optimal control
problems, especially stochastic maximum principles (SMPs), for mean-field systems
have become a popular topic; see, for example,
\cite{Li2012,Yong2013,WZZ2014,CDL2016,CZ2016,CD2018,ZSX2018,LSX2019,ABC2019,WangWu2022}.

Another feature of this paper is the use of a regime switching model, in which
the \emph{continuous} state of the LQ problem and the \emph{discrete} state of
the Markov chain coexist; see
\cite{Zhang2001,ZhouYin2003,SSZ2011,Zhu2011,ZLX2021,LWZ2022,LX2022}
for more information and applications of regime switching models.
Recently, Nguyen, Yin, and Hoang \cite{NYH2020} established the law of large numbers
for systems with regime switching and mean-field interactions, where the mean-field
limit was characterized as the \emph{conditional expectation} of the solution to
a conditional mean-field SDE with regime switching (see also Remark \ref{mean field limit}).
This work paves the way for treating mean-field optimal control problems with
regime switching; see \cite{NNY2020,NYN2021,BDTY2020,JLSY}.

In this paper, we consider an LQ leader-follower stochastic differential game
for mean-field switching diffusions. Based on the SMP in Nguyen, Nguyen, and
Yin \cite{NNY2020}, an open-loop optimal control for the follower is obtained.
Then, by applying the \emph{four-step scheme} developed by Ma, Protter, and Yong
\cite{MPY1994}, we derive its (anticipating) state feedback representation
in terms of two Riccati equations and an auxiliary BSDE. Knowing the follower's
optimal control, the leader faces a state equation which is a conditional mean-field
FBSDE with regime switching. We also utilize the SMP to obtain an open-loop optimal
control for the leader. Then, by the \emph{dimensional augmentation approach}
in Yong \cite{Yong2002}, a \emph{non-anticipating} state feedback representation
is derived in terms of two Riccati equations. As a consequence, the follower's
optimal control can be also represented in a non-anticipating way.

The rest of this paper is organized as follows. In the next, we present an example
which motivates us to study the leader-follower problem in this paper.
In Section \ref{Problem formulation}, we formulate the problem and provide some
preliminary results. In Sections \ref{follower problem} and \ref{leader problem},
we solve the LQ problems for the follower and the leader, respectively. Finally,
Section \ref{conclusion} concludes the paper.

\emph{Motivation: a pension fund optimization problem}.
Typically, in a defined benefit (DB) scheme pension fund there are two participants
who make contributions: one is the leader (such as the company) with contribution
rate $u_{2}(\cdot)$, the other one is the follower (such as the individual) with
contribution rate $u_{1}(\cdot)$. The dynamics of the pension fund is described as
\begin{equation*}
\begin{aligned}
dF(t)=F(t)d\Delta(t)+\{u_{1}(t)+u_{2}(t)-\xi_{0}\}dt,
\end{aligned}
\end{equation*}
where $d\Delta(t)$ is the return rate of the fund and $\xi_{0}$ is the
pension scheme benefit outgo. The pension fund is invested in a bond
$S_{0}(t)$ and a stock $S(t)$, which are given by
\begin{equation*}
\left\{
\begin{aligned}
dS_{0}(t)=&r(\alpha(t))S_{0}(t)dt,\\
dS(t)=&b(\alpha(t))S(t)dt+\sigma(\alpha(t))S(t)dW(t),
\end{aligned}
\right.
\end{equation*}
where $r(i)$ is the interest rate, $b(i)$ is the appreciation rate,
and $\sigma(i)$ is the volatility corresponding to the market regime
$\alpha(t)=i$. Assume the proportions $\pi(\cdot)$ and $1-\pi(\cdot)$
of the fund are to be allocated in the stock and the bond, respectively.
Then we have
\begin{equation*}
\begin{aligned}
d\Delta(t)=\{r(\alpha(t))+[b(\alpha(t))-r(\alpha(t))]\pi(t)\}dt
+\sigma(\alpha(t))\pi(t)dW(t).
\end{aligned}
\end{equation*}
Therefore, the dynamics of the pension fund can be written as
\begin{equation*}
\begin{aligned}
dF(t)=\{r(\alpha(t))&F(t)+[b(\alpha(t))-r(\alpha(t))]\pi(t)F(t)\\
+&u_{1}(t)+u_{2}(t)-\xi_{0}\}dt+\sigma(\alpha(t))\pi(t)F(t)dW(t).
\end{aligned}
\end{equation*}
The cost functionals for the follower and the leader to minimize
are defined as
\begin{equation*}
\begin{aligned}
J_{k}(u_{1}(\cdot),u_{2}(\cdot))
=\frac{1}{2}E\bigg[\int_{0}^{T}\Big(u_{k}(t)-\xi_{k}\Big)^{2}dt
+\Big(E[F(T)|\mathcal{F}_{T}^{\alpha}]-\xi_{T}\Big)^{2}\bigg],\quad k=1,2,
\end{aligned}
\end{equation*}
respectively, where $\xi_{k}$, $k=1,2$, are the running benchmark, and $\xi_{T}$
is the terminal wealth target; both are introduced to measure the stability and
performance of the pension scheme.

The above pension fund optimization problem formulates naturally a special case of
the LQ leader-follower game considered in this paper. For more pension fund
optimization problems under various contexts, see \cite{JR2001,HWX2009,ZhengShi2020};
for a conditional mean-variance portfolio selection problem (as an application of
conditional mean-field control theory), see \cite{NYN2021}.

\section{Problem formulation and preliminaries}\label{Problem formulation}

Let $R^{n}$ be the $n$-dimensional Euclidean space with Euclidean norm $|\cdot|$
and Euclidean inner product $\langle\cdot,\cdot\rangle$. Let $R^{n\times m}$ be
the space of all $(n\times m)$ matrices. $A^{\top}$ denotes the transpose of
a vector or matrix $A$.
$I_{n}$ denotes the $(n\times n)$ identity matrix.

Let $[0,T]$ be a finite time horizon and $(\Omega,\mathcal{F},P)$ be a fixed
probability space on which a one-dimensional standard Brownian motion $W(t)$,
$t\in [0,T]$, and a Markov chain $\alpha(t)$, $t\in [0,T]$, are defined.
The Markov chain $\alpha(\cdot)$ takes values in a finite state space $\mathcal{M}$.
Let $Q=(\lambda_{ij})_{i,j\in\mathcal{M}}$ be the generator (i.e., the matrix of
transition rates) of $\alpha(\cdot)$ with $\lambda_{ij}\geq 0$ for $i\neq j$ and
$\sum_{j\in \mathcal{M}}\lambda_{ij}=0$ for each $i\in \mathcal{M}$. Assume that
$W(\cdot)$ and $\alpha(\cdot)$ are independent. For $t\geq0$, denote
$\mathcal{F}^{\alpha}_{t}=\sigma\{\alpha(s):0\leq s\leq t\}$
and $\mathcal{F}_{t}=\sigma\{W(s),\alpha(s):0\leq s\leq t\}$.
Let $\mathcal{L}_{\mathcal{F}}^{2}(R^{n})$ be the set of all $R^{n}$-valued
$\mathcal{F}_{t}$-adapted processes $x(\cdot)$ on $[0,T]$ such that
$E\int_{0}^{T}|x(t)|^{2}dt<\infty$.

The state of the system is described by the following linear conditional
mean-field SDE with regime switching on $[0,T]$:
\begin{equation}\label{system}
\left\{
\begin{aligned}
dx(t)=&\Big[A(\alpha(t))x(t)+\widehat{A}(\alpha(t))E[x(t)|\mathcal{F}_{t}^{\alpha}]
+B_{1}(\alpha(t))u_{1}(t)+B_{2}(\alpha(t))u_{2}(t)\Big]dt\\
+&\Big[C(\alpha(t))x(t)+\widehat{C}(\alpha(t))E[x(t)|\mathcal{F}_{t}^{\alpha}]
+D_{1}(\alpha(t))u_{1}(t)+D_{2}(\alpha(t))u_{2}(t)\Big]dW(t),\\
x(0)=&x_{0},
\end{aligned}
\right.
\end{equation}
where $x(\cdot)$ is the state process with values in $R^{n}$, $u_{1}(\cdot)$ and
$u_{2}(\cdot)$ are control processes taken by the follower and the leader, with
values in $R^{m_{1}}$ and $R^{m_{2}}$, respectively, and $A(i)$, $\widehat{A}(i)$,
$B_{1}(i)$, $B_{2}(i)$, $C(i)$, $\widehat{C}(i)$, $D_{1}(i)$, $D_{2}(i)$,
$i\in\mathcal{M}$, are constant matrices of suitable dimensions. It follows from
Nguyen et al. \cite{NNY2020,NYN2021} that,
for any $u_{1}(\cdot)\in \mathcal{L}_{\mathcal{F}}^{2}(R^{m_{1}})$
and $u_{2}(\cdot)\in \mathcal{L}_{\mathcal{F}}^{2}(R^{m_{2}})$,
SDE (\ref{system}) admits a unique solution
$x(\cdot)\in \mathcal{L}_{\mathcal{F}}^{2}(R^{n})$.
Then, $\mathcal{U}_{1}\doteq \mathcal{L}_{\mathcal{F}}^{2}(R^{m_{1}})$
and $\mathcal{U}_{2}\doteq \mathcal{L}_{\mathcal{F}}^{2}(R^{m_{2}})$
are called the admissible control sets for the follower and the leader, respectively.

The cost functionals for the follower and the leader to minimize are defined as
\begin{equation}\label{cost functionals}
\begin{aligned}
J_{k}(u_{1}(\cdot),u_{2}(\cdot))
=&\frac{1}{2}E\bigg[\int_{0}^{T}\bigg(\Big\langle Q_{k}(\alpha(t))x(t),x(t)\Big\rangle
+\Big\langle \widehat{Q}_{k}(\alpha(t))E[x(t)|\mathcal{F}_{t}^{\alpha}],E[x(t)|\mathcal{F}_{t}^{\alpha}]\Big\rangle\\
&+\Big\langle N_{k}(\alpha(t))u_{k}(t),u_{k}(t)\Big\rangle\bigg)dt+\Big\langle G_{k}(\alpha(T))x(T),x(T)\Big\rangle\\
&+\Big\langle \widehat{G}_{k}(\alpha(T))E[x(T)|\mathcal{F}_{T}^{\alpha}],E[x(T)|\mathcal{F}_{T}^{\alpha}]\Big\rangle\bigg],
\quad k=1,2,
\end{aligned}
\end{equation}
respectively, where $Q_{k}(i)$, $\widehat{Q}_{k}(i)$, $N_{k}(i)$, $G_{k}(i)$, $\widehat{G}_{k}(i)$,
$k=1,2$, $i\in\mathcal{M}$, are constant symmetric matrices of suitable dimensions.
\begin{remark}\label{mean field limit}
In fact, SDE (\ref{system}) is obtained as the mean-square limit as $N\rightarrow\infty$
of a system of interacting particles in the following form:
\begin{equation*}
\left\{
\begin{aligned}
dx^{l,N}(t)=&\bigg[A(\alpha(t))x^{l,N}(t)+\widehat{A}(\alpha(t))\frac{1}{N}\sum_{l=1}^{N}x^{l,N}(t)
+B_{1}(\alpha(t))u_{1}(t)+B_{2}(\alpha(t))u_{2}(t)\bigg]dt\\
+\bigg[C(&\alpha(t))x^{l,N}(t)+\widehat{C}(\alpha(t))\frac{1}{N}\sum_{l=1}^{N}x^{l,N}(t)
+D_{1}(\alpha(t))u_{1}(t)+D_{2}(\alpha(t))u_{2}(t)\bigg]dW^{l}(t),\\
x^{l,N}(0)=&x_{0},\quad 1\leq l\leq N,
\end{aligned}
\right.
\end{equation*}
where $\{W^{l}(\cdot)\}_{l=1}^{N}$ is a collection of independent standard Brownian motions
and the Markov chain $\alpha(\cdot)$ serves as a common noise for all particles, which leads
to the conditional expectations rather than expectations in (\ref{system}).

Intuitively, since all the particles depend on the history of $\alpha(\cdot)$, their average
and thereby its limit as $N\rightarrow\infty$ should also depend on this process. This intuition
has been justified by the law of large numbers established by Nguyen et al. \cite[Theorem 2.1]{NYH2020},
which shows that the joint process $(\frac{1}{N}\sum_{l=1}^{N}x^{l,N}(\cdot),\alpha(\cdot))$
converges weakly to a process $(\mu_{\alpha}(\cdot),\alpha(\cdot))$, where
$(\mu_{\alpha}(t),\alpha(t))=(E[x(t)|\mathcal{F}_{t}^{\alpha}],\alpha(t))$, $0\leq t\leq T$,
and $x(\cdot)$ is exactly the solution of (\ref{system}).
\end{remark}
\begin{remark}
Note that the cost functionals $J_{k}$, $k=1,2$, defined by (\ref{cost functionals})
are standard in the LQ mean-field control literature (see \cite{Yong2013,NNY2020,NYN2021})
and, if we assume the Assumptions (A1) and (A2) given in Sections \ref{follower problem}
and \ref{leader problem} hold, then $J_{k}$ is convex with respect to $u_{k}$, $k=1,2$,
respectively. However, for LQ mean-field games of large-population systems, the tracking-type
cost functionals where one wants to keep the system states stay as much close as possible
to a function of the mean-field term are more frequently adopted (see \cite{HCM2007,LZ2008,CZ2016}).
\end{remark}
Now we explain the leader-follower feature of the game; see also Yong \cite{Yong2002}.
In the game, for any $u_{2}(\cdot)\in \mathcal{U}_{2}$ of the leader, the follower
would like to choose an optimal control $u_{1}^{*}(\cdot)\in \mathcal{U}_{1}$ so that
$J_{1}(u_{1}^{*}(\cdot),u_{2}(\cdot))$ achieves the minimum of $J_{1}(u_{1}(\cdot),u_{2}(\cdot))$
over $u_{1}(\cdot)\in \mathcal{U}_{1}$. Knowing the follower's optimal control $u_{1}^{*}(\cdot)$
(depending on $u_{2}(\cdot)$), the leader would like to choose an optimal control
$u_{2}^{*}(\cdot)\in \mathcal{U}_{2}$ to minimize $J_{2}(u_{1}^{*}(\cdot),u_{2}(\cdot))$
over $u_{2}(\cdot)\in \mathcal{U}_{2}$.

In a more rigorous way, the follower wants to find an optimal map
$\Pi_{1}^{*}:\mathcal{U}_{2}\mapsto\mathcal{U}_{1}$ and the leader
wants to find an optimal control $u_{2}^{*}(\cdot)\in \mathcal{U}_{2}$ such that
\begin{equation*}
\left\{
\begin{aligned}
J_{1}(\Pi_{1}^{*}[u_{2}(\cdot)](\cdot),u_{2}(\cdot))
=&\inf_{u_{1}(\cdot)\in \mathcal{U}_{1}}J_{1}(u_{1}(\cdot),u_{2}(\cdot)),
\quad\forall u_{2}(\cdot)\in\mathcal{U}_{2},\\
J_{2}(\Pi_{1}^{*}[u_{2}^{*}(\cdot)](\cdot),u_{2}^{*}(\cdot))
=&\inf_{u_{2}(\cdot)\in \mathcal{U}_{2}}J_{2}(\Pi_{1}^{*}[u_{2}(\cdot)](\cdot),u_{2}(\cdot)).
\end{aligned}
\right.
\end{equation*}
If the above optimal pair $(\Pi_{1}^{*}[\cdot],u_{2}^{*}(\cdot))$ exists, it is called
an \emph{open-loop} Stackelberg equilibrium of the leader-follower stochastic differential game.

Then we present some preliminary results on the martingales associated with a Markov chain,
which are needed to establish the conditional mean-field BSDEs with regime switching.
For each pair $(i,j)\in \mathcal{M}\times \mathcal{M}$ with $i\neq j$, define
$[M_{ij}](t)=\sum_{0\leq s\leq t}1_{\{\alpha(s-)=i\}}1_{\{\alpha(s)=j\}}$
and $\langle M_{ij}\rangle(t)=\int_{0}^{t}\lambda_{ij}1_{\{\alpha(s-)=i\}}ds$,
where $1_{A}$ denotes the indicator function of a set $A$.
It follows from \cite{NNY2020,NYN2021} that the process
$M_{ij}(t)\doteq[M_{ij}](t)-\langle M_{ij}\rangle(t)$
is a purely discontinuous and square-integrable martingale with respect to
$\mathcal{F}_{t}^{\alpha}$, which is null at the origin. In this sense,
$[M_{ij}](t)$ and $\langle M_{ij}\rangle(t)$ are the optional and predictable
quadratic variations of $M_{ij}(t)$, respectively. In addition, we denote
$M_{ii}(t)=[M_{ii}](t)=\langle M_{ii}\rangle(t)\equiv0$ for each $i\in\mathcal{M}$.

Let $\mathcal{S}_{\mathcal{F}}^{2}(R^{n})$ be the set of all $R^{n}$-valued
$\mathcal{F}_{t}$-adapted c\`{a}dl\`{a}g processes $y(\cdot)$ on $[0,T]$ such that
$E\int_{0}^{T}|y(t)|^{2}dt<\infty$. Let $\mathcal{K}_{\mathcal{F}}^{2}(R^{n})$
be the set of all collections of $R^{n}$-valued $\mathcal{F}_{t}$-adapted processes
$\{k_{ij}(\cdot)\}_{i,j\in \mathcal{M}}$ on $[0,T]$ such that
$\sum_{i,j\in \mathcal{M}}E\int_{0}^{T}|k_{ij}(t)|^{2}d[M_{ij}](t)<\infty$
with $k_{ii}(t)\equiv0$ for each $i\in \mathcal{M}$. For convenience, we also denote
$k(\cdot)=\{k_{ij}(\cdot)\}_{i,j\in \mathcal{M}}$ and
\begin{equation*}
\begin{aligned}
\int_{0}^{t}k(s)\bullet dM(s)=\sum_{i,j\in \mathcal{M}}\int_{0}^{t}k_{ij}(s)dM_{ij}(s),\quad
k(s)\bullet dM(s)=\sum_{i,j\in \mathcal{M}}k_{ij}(s)dM_{ij}(s).
\end{aligned}
\end{equation*}
The following two lemmas play an important role in the subsequent analysis.
The proof of the first lemma is elementary and the proof of the second one
is similar to that of Xiong \cite[Lemma 5.4]{Xiong2008}. For completeness
and readers' convenience, their proofs are provided here.
\begin{lemma}\label{Lemma conditional E}
For any $\mathcal{F}_{t}$-adapted and square-integrable processes
$x(\cdot)$ and $y(\cdot)$, we have
\begin{equation*}
\begin{aligned}
E\Big[x(t)E[y(t)|\mathcal{F}_{t}^{\alpha}]\Big]
=E\Big[E[x(t)|\mathcal{F}_{t}^{\alpha}]y(t)\Big]
=E\Big[E[x(t)|\mathcal{F}_{t}^{\alpha}]E[y(t)|\mathcal{F}_{t}^{\alpha}]\Big].
\end{aligned}
\end{equation*}
\end{lemma}
\begin{proof}
Note that
\begin{equation*}
\begin{aligned}
E\Big[x(t)E[y(t)|\mathcal{F}_{t}^{\alpha}]\Big]
=E\Big[E\Big(x(t)E[y(t)|\mathcal{F}_{t}^{\alpha}]\Big|\mathcal{F}_{t}^{\alpha}\Big)\Big]
=E\Big[E[x(t)|\mathcal{F}_{t}^{\alpha}]E[y(t)|\mathcal{F}_{t}^{\alpha}]\Big].
\end{aligned}
\end{equation*}
Similarly,
\begin{equation*}
\begin{aligned}
E\Big[E[x(t)|\mathcal{F}_{t}^{\alpha}]y(t)\Big]
=E\Big[E[x(t)|\mathcal{F}_{t}^{\alpha}]E[y(t)|\mathcal{F}_{t}^{\alpha}]\Big].
\end{aligned}
\end{equation*}
Consequently, the desired conclusion follows.
\end{proof}
\begin{lemma}\label{Lemma filtering}
For any $\mathcal{F}_{t}$-adapted and square-integrable process $x(\cdot)$,
we have
\begin{equation*}
\begin{aligned}
E\bigg[\int_{0}^{t}x(s)ds\bigg|\mathcal{F}_{t}^{\alpha}\bigg]
=\int_{0}^{t}E[x(s)|\mathcal{F}_{s}^{\alpha}]ds,
\end{aligned}
\end{equation*}
and
\begin{equation*}
\begin{aligned}
E\bigg[\int_{0}^{t}x(s)dW(s)\bigg|\mathcal{F}_{t}^{\alpha}\bigg]=0.
\end{aligned}
\end{equation*}
\end{lemma}
\begin{proof}
For the first equation, from the Markov property of $\alpha(\cdot)$
and the independence of $W(\cdot)$ and $\alpha(\cdot)$,
it follows that
\begin{equation*}
\begin{aligned}
E\bigg[\int_{0}^{t}x(s)ds\bigg|\mathcal{F}_{t}^{\alpha}\bigg]
=\int_{0}^{t}E[x(s)|\mathcal{F}_{t}^{\alpha}]ds
=\int_{0}^{t}E[x(s)|\mathcal{F}_{s}^{\alpha}]ds.
\end{aligned}
\end{equation*}
Now we proceed to prove the second equation. We first suppose $x(\cdot)$ is simple, namely
\begin{equation*}
\begin{aligned}
x(s)=\sum_{m\geq1}x_{m}1_{[t_{m},t_{m+1})}(s),
\end{aligned}
\end{equation*}
where, for each $m\geq1$, $x_{m}$ is an $\mathcal{F}_{t_{m}}$-measurable random variable.
As $W(t_{m+1})-W(t_{m})$ is independent of
$\mathcal{F}_{t}^{\alpha}\vee\sigma(x_{m})\doteq\sigma(\mathcal{F}_{t}^{\alpha}\cup\sigma(x_{m}))$,
we have
\begin{equation*}
\begin{aligned}
E\bigg[\int_{0}^{t}x(s)dW(s)\bigg|\mathcal{F}_{t}^{\alpha}\bigg]
=&\sum_{m\geq1}E\bigg[x_{m}[W(t_{m+1})-W(t_{m})]\bigg|\mathcal{F}_{t}^{\alpha}\bigg]\\
=&\sum_{m\geq1}E\bigg[x_{m}E\bigg(W(t_{m+1})-W(t_{m})
\bigg|\mathcal{F}_{t}^{\alpha}\vee\sigma(x_{m})\bigg)\bigg|\mathcal{F}_{t}^{\alpha}\bigg]=0.
\end{aligned}
\end{equation*}
For general $x(\cdot)$, we can approximate $x(\cdot)$ by a sequence of simple processes
$\{x_{n}(\cdot):n\geq1\}$ such that $|x_{n}(s)|\leq|x(s)|$,
a.s., for each $n\geq1$ and all $s\leq t$. Note that
\begin{equation*}
\begin{aligned}
E\bigg[\bigg|\int_{0}^{t}x_{n}(s)dW(s)\bigg|^{2}\bigg]
=E\bigg[\int_{0}^{t}|x_{n}(s)|^{2}ds\bigg]
\leq E\bigg[\int_{0}^{t}|x(s)|^{2}ds\bigg]<\infty,
\end{aligned}
\end{equation*}
which implies that $\{\int_{0}^{t}x_{n}(s)dW(s):n\geq1\}$ is uniformly integrable. Therefore,
\begin{equation*}
\begin{aligned}
E\bigg[\int_{0}^{t}x(s)dW(s)\bigg|\mathcal{F}_{t}^{\alpha}\bigg]
=\lim_{n\rightarrow\infty}
E\bigg[\int_{0}^{t}x_{n}(s)dW(s)\bigg|\mathcal{F}_{t}^{\alpha}\bigg]
=0.
\end{aligned}
\end{equation*}
This completes the proof.
\end{proof}

\section{The problem for the follower}\label{follower problem}

In this section, we deal with the problem for the follower.
For convenience, we denote
\begin{equation*}
\begin{aligned}
\widehat{\phi}(t)=E[\phi(t)|\mathcal{F}_{t}^{\alpha}],
\end{aligned}
\end{equation*}
for a process $\phi(\cdot)$. We will apply the SMP obtained by Nguyen et al.
\cite[Theorem 3.7]{NNY2020} to solve the follower's problem. Besides the
open-loop optimal control, we would like further to find its state feedback
representation. We make the following assumption:

(A1) $Q_{1}(i)\geq0$, $\widehat{Q}_{1}(i)\geq0$, $N_{1}(i)>0$,
$G_{1}(i)\geq0$, $\widehat{G}_{1}(i)\geq0$, $i\in\mathcal{M}$.
\begin{lemma}
Let Assumption (A1) hold. For any given $u_{2}(\cdot)\in\mathcal{U}_{2}$
for the leader, let $u_{1}^{*}(\cdot)$ be an optimal control for the follower,
then $u_{1}^{*}(\cdot)$ should have the following form:
\begin{equation}\label{follower optimal control}
\begin{aligned}
u_{1}^{*}(t)=-\widetilde{N}_{1}^{-1}(t,\alpha(t))
\Big[S_{1}(t,\alpha(t))x(t)+\widehat{S}_{1}(t,\alpha(t))\widehat{x}(t)+\Phi(t)\Big],
\end{aligned}
\end{equation}
where, for notational simplicity, we denote
\begin{equation*}
\begin{aligned}
&\widetilde{N}_{1}(t,i)=N_{1}(i)+D_{1}^{\top}(i)P_{1}(t,i)D_{1}(i),\\
&S_{1}(t,i)=B_{1}^{\top}(i)P_{1}(t,i)+D_{1}^{\top}(i)P_{1}(t,i)C(i),\quad
\widehat{S}_{1}(t,i)=B_{1}^{\top}(i)\widehat{P}_{1}(t,i)+D_{1}^{\top}(i)P_{1}(t,i)\widehat{C}(i),\\
&\Phi(t)=B_{1}^{\top}(\alpha(t))\varphi(t)+D_{1}^{\top}(\alpha(t))\theta(t)
+D_{1}^{\top}(\alpha(t))P_{1}(t,\alpha(t))D_{2}(\alpha(t))u_{2}(t),\quad i\in\mathcal{M},
\end{aligned}
\end{equation*}
and $P_{1}(\cdot,i)$ and $\widehat{P}_{1}(\cdot,i)$, $i\in\mathcal{M}$, are the solutions
of Riccati equations (\ref{follower 1}) and (\ref{follower 2}), respectively, and
$(\varphi(\cdot),\theta(\cdot),\eta(\cdot))\in \mathcal{S}_{\mathcal{F}}^{2}(R^{n})
\times \mathcal{L}_{\mathcal{F}}^{2}(R^{n})\times \mathcal{K}_{\mathcal{F}}^{2}(R^{n})$
is the solution of BSDE (\ref{follower 3}).
\end{lemma}
\begin{proof}
From \cite[Theorem 3.7]{NNY2020}, the adjoint equation for the follower is given by
\begin{equation}\label{follower adjoint equation}
\left\{
\begin{aligned}
dp(t)=&-\Big[A^{\top}(\alpha(t))p(t)+\widehat{A}^{\top}(\alpha(t))\widehat{p}(t)
+C^{\top}(\alpha(t))q(t)+\widehat{C}^{\top}(\alpha(t))\widehat{q}(t)\\
&+Q_{1}(\alpha(t))x(t)+\widehat{Q}_{1}(\alpha(t))\widehat{x}(t)\Big]dt
+q(t)dW(t)+r(t)\bullet dM(t),\\
p(T)=&G_{1}(\alpha(T))x(T)+\widehat{G}_{1}(\alpha(T))\widehat{x}(T),
\end{aligned}
\right.
\end{equation}
which, from \cite[Theorem 3.4]{NNY2020}, admits a unique solution
$(p(\cdot),q(\cdot),r(\cdot))\in\mathcal{S}_{\mathcal{F}}^{2}(R^{n})
\times \mathcal{L}_{\mathcal{F}}^{2}(R^{n})\times \mathcal{K}_{\mathcal{F}}^{2}(R^{n})$,
and an optimal control $u_{1}^{*}(\cdot)$ for the follower should satisfy
\begin{equation}\label{follower open loop}
\begin{aligned}
N_{1}(\alpha(t))u_{1}^{*}(t)+B_{1}^{\top}(\alpha(t))p(t)
+D_{1}^{\top}(\alpha(t))q(t)=0.
\end{aligned}
\end{equation}
Inspired by the terminal condition of the adjoint equation
(\ref{follower adjoint equation}), it is natural to guess
\begin{equation}\label{follower FSS}
\begin{aligned}
p(t)=P_{1}(t,\alpha(t))x(t)
+\widehat{P}_{1}(t,\alpha(t))\widehat{x}(t)+\varphi(t),
\end{aligned}
\end{equation}
for some $R^{n\times n}$-valued deterministic, differentiable, and symmetric
functions $P_{1}(t,i)$ and  $\widehat{P}_{1}(t,i)$,\  $i\in \mathcal{M}$, 	
and an $R^{n}$-valued $\mathcal{F}_{t}$-adapted process $\varphi(t)$ with
\begin{equation*}
\begin{aligned}
d\varphi(t)=\gamma(t)dt+\theta(t)dW(t)+\eta(t)\bullet dM(t).
\end{aligned}
\end{equation*}
Then,
\begin{equation}\label{follower p hat}
\begin{aligned}
\widehat{p}(t)=\Big(P_{1}(t,\alpha(t))
+\widehat{P}_{1}(t,\alpha(t))\Big)\widehat{x}(t)
+\widehat{\varphi}(t).
\end{aligned}
\end{equation}
From Lemma \ref{Lemma filtering}, we have
\begin{equation*}
\begin{aligned}
d\widehat{x}(t)=&\Big[\Big(A(\alpha(t))+\widehat{A}(\alpha(t))\Big)\widehat{x}(t)
+B_{1}(\alpha(t))\widehat{u}_{1}(t)+B_{2}(\alpha(t))\widehat{u}_{2}(t)\Big]dt.
\end{aligned}
\end{equation*}
In the rest of this paper, the arguments $t$ and $\alpha(t)$ will be dropped to save space,
if needed and when no confusion arises.	Applying It\^{o}'s formula for Markov-modulated
processes (see Zhou and Yin \cite[Lemma 3.1]{ZhouYin2003}) to (\ref{follower FSS}), we obtain
\begin{equation}\label{follower dp}
\begin{aligned}
dp=&\bigg(\dot{P}_{1}+\sum_{j\in\mathcal{M}}\lambda_{\alpha(t),j}[P_{1}(t,j)-P_{1}(t,\alpha(t))]\bigg)xdt
+\sum_{i,j\in\mathcal{M}}[P_{1}(t,j)-P_{1}(t,i)]xdM_{ij}\\
&+P_{1}[Ax+\widehat{A}\widehat{x}+B_{1}u_{1}+B_{2}u_{2}]dt+P_{1}[Cx+\widehat{C}\widehat{x}+D_{1}u_{1}+D_{2}u_{2}]dW\\
&+\bigg(\dot{\widehat{P}}_{1}
+\sum_{j\in\mathcal{M}}\lambda_{\alpha(t),j}[\widehat{P}_{1}(t,j)-\widehat{P}_{1}(t,\alpha(t))]\bigg)\widehat{x}dt
+\sum_{i,j\in\mathcal{M}}[\widehat{P}_{1}(t,j)-\widehat{P}_{1}(t,i)]\widehat{x}dM_{ij}\\
&+\widehat{P}_{1}\Big[(A+\widehat{A})\widehat{x}+B_{1}\widehat{u}_{1}+B_{2}\widehat{u}_{2}\Big]dt
+\gamma dt+\theta dW+\eta\bullet dM.
\end{aligned}
\end{equation}
Comparing the coefficients of $dW$ parts in (\ref{follower adjoint equation})
and (\ref{follower dp}), it follows that
\begin{equation}\label{follower q}
\begin{aligned}
q=P_{1}\Big[Cx+\widehat{C}\widehat{x}+D_{1}u_{1}+D_{2}u_{2}\Big]+\theta,
\end{aligned}
\end{equation}
and then,
\begin{equation}\label{follower q hat}
\begin{aligned}
\widehat{q}=P_{1}\Big[(C+\widehat{C})\widehat{x}+D_{1}\widehat{u}_{1}+D_{2}\widehat{u}_{2}\Big]+\widehat{\theta}.
\end{aligned}
\end{equation}
Inserting (\ref{follower FSS}) and (\ref{follower q}) into (\ref{follower open loop}) yields
\begin{equation*}
\begin{aligned}
0=&\Big(N+D_{1}^{\top}P_{1}D_{1}\Big)u_{1}^{*}
+\Big(B_{1}^{\top}P_{1}+D_{1}^{\top}P_{1}C\Big)x
+\Big(B_{1}^{\top}\widehat{P}_{1}+D_{1}^{\top}P_{1}\widehat{C}\Big)\widehat{x}\\
&+B_{1}^{\top}\varphi+D_{1}^{\top}\theta+D_{1}^{\top}P_{1}D_{2}u_{2},
\end{aligned}
\end{equation*}
i.e., $u_{1}^{*}=-\widetilde{N}_{1}^{-1}[S_{1}x+\widehat{S}_{1}\widehat{x}+\Phi]$,
provided $\widetilde{N}_{1}$ is invertible. So we have (\ref{follower optimal control}). Also,
\begin{equation}\label{follower optimal control hat}
\begin{aligned}
\widehat{u}_{1}^{*}=-\widetilde{N}_{1}^{-1}\Big[(S_{1}+\widehat{S}_{1})\widehat{x}+\widehat{\Phi}\Big].
\end{aligned}
\end{equation}
On the one hand, substituting (\ref{follower FSS}), (\ref{follower p hat}),
(\ref{follower q}), (\ref{follower q hat}), and (\ref{follower optimal control}),
(\ref{follower optimal control hat}) into (\ref{follower adjoint equation}),
we have
\begin{equation}\label{follower comparison 1}
\begin{aligned}
dp=&-\Big[\Big(A^{\top}P_{1}+C^{\top}P_{1}C-C^{\top}P_{1}D_{1}\widetilde{N}_{1}^{-1}S_{1}+Q_{1}\Big)x\\
&+\Big(\widehat{A}^{\top}P_{1}+(A+\widehat{A})^{\top}\widehat{P}_{1}
+C^{\top}P_{1}\widehat{C}+\widehat{C}^{\top}P_{1}C+\widehat{C}^{\top}P_{1}\widehat{C}\\
&-C^{\top}P_{1}D_{1}\widetilde{N}_{1}^{-1}\widehat{S}_{1}
-\widehat{C}^{\top}P_{1}D_{1}\widetilde{N}_{1}^{-1}(S_{1}+\widehat{S}_{1})+\widehat{Q}_{1}\Big)\widehat{x}\\
&+\Big(A-B_{1}\widetilde{N}_{1}^{-1}D_{1}^{\top}P_{1}C\Big)^{\top}\varphi
+\Big(\widehat{A}-B_{1}\widetilde{N}_{1}^{-1}D_{1}^{\top}P_{1}\widehat{C}\Big)^{\top}\widehat{\varphi}\\
&+\Big(C-D_{1}\widetilde{N}_{1}^{-1}D_{1}^{\top}P_{1}C\Big)^{\top}\theta
+\Big(\widehat{C}-D_{1}\widetilde{N}_{1}^{-1}D_{1}^{\top}P_{1}\widehat{C}\Big)^{\top}\widehat{\theta}\\
&+\Big(C^{\top}P_{1}D_{2}-C^{\top}P_{1}D_{1}\widetilde{N}_{1}^{-1}D_{1}^{\top}P_{1}D_{2}\Big)u_{2}\\
&+\Big(\widehat{C}^{\top}P_{1}D_{2}
-\widehat{C}^{\top}P_{1}D_{1}\widetilde{N}_{1}^{-1}D_{1}^{\top}P_{1}D_{2}\Big)\widehat{u}_{2}\Big]dt\\
&+qdW+r\bullet dM.
\end{aligned}
\end{equation}
On the other hand, substituting (\ref{follower optimal control})
and \eqref{follower optimal control hat} into (\ref{follower dp}),
we have
\begin{equation}\label{follower comparison 2}
\begin{aligned}
dp=&\Big[\Big(\dot{P}_1+P_{1}A-P_{1}B_{1}\widetilde{N}_{1}^{-1}S_{1}
+\sum_{j\in\mathcal{M}}\lambda_{\alpha(t),j}[P_{1}(t,j)-P_{1}(t,\alpha(t))]\Big)x\\
&+\Big(\dot{\widehat{P}}_1+P_{1}\widehat{A}+\widehat{P}_{1}(A+\widehat{A})
-P_{1}B_{1}\widetilde{N}_{1}^{-1}\widehat{S}_{1}
-\widehat{P}_{1}B_{1}\widetilde{N}_{1}^{-1}(S_{1}+\widehat{S}_{1})\\
&+\sum_{j\in\mathcal{M}}\lambda_{\alpha(t),j}[\widehat{P}_{1}(t,j)-\widehat{P}_{1}(t,\alpha(t))]\Big)\widehat{x}\\
&+\gamma-P_{1}B_{1}\widetilde{N}_{1}^{-1}B_{1}^{\top}\varphi
-\widehat{P}_{1}B_{1}\widetilde{N}_{1}^{-1}B_{1}^{\top}\widehat{\varphi}
-P_{1}B_{1}\widetilde{N}_{1}^{-1}D_{1}^{\top}\theta
-\widehat{P}_{1}B_{1}\widetilde{N}_{1}^{-1}D_{1}^{\top}\widehat{\theta}\\
&+\Big(P_{1}B_{2}-P_{1}B_{1}\widetilde{N}_{1}^{-1}D_{1}^{\top}P_{1}D_{2}\Big)u_{2}
+\Big(\widehat{P}_{1}B_{2}-\widehat{P}_{1}B_{1}\widetilde{N}_{1}^{-1}D_{1}^{\top}P_{1}D_{2}\Big)\widehat{u}_{2}\Big]dt\\
&+\Big\{\cdots\Big\}dW+\Big\{\cdots\Big\}\bullet dM.
\end{aligned}
\end{equation}
By equalizing the coefficients of $x$ and $\widehat{x}$
as well as the non-homogeneous terms in the $dt$ parts of
(\ref{follower comparison 1}) and (\ref{follower comparison 2}),
we obtain two Riccati equations:
\begin{equation}\label{follower 1}
\left\{
\begin{aligned}
\dot{P}_{1}(t,i)
=&-\Big[P_{1}(t,i)A(i)+A^{\top}(i)P_{1}(t,i)
+C^{\top}(i)P_{1}(t,i)C(i)+Q_{1}(i)\\
&-S_{1}^{\top}(t,i)\widetilde{N}_{1}^{-1}(t,i)S_{1}(t,i)
+\sum_{j\in\mathcal{M}}\lambda_{ij}[P_{1}(t,j)-P_{1}(t,i)]\Big],\\
P_{1}(T,i)=&G_{1}(i),\quad i\in\mathcal{M},
\end{aligned}
\right.
\end{equation}
and
\begin{equation}\label{follower 2}
\left\{
\begin{aligned}
\dot{\widehat{P}}_{1}(t,i)
=&-\Big[\widehat{P}_{1}(t,i)(A(i)+\widehat{A}(i))
+(A(i)+\widehat{A}(i))^{\top}\widehat{P}_{1}(t,i)\\
&+P_{1}(t,i)\widehat{A}(i)+\widehat{A}^{\top}(i)P_{1}(t,i)
+C^{\top}(i)P_{1}(t,i)\widehat{C}(i)\\
&+\widehat{C}^{\top}(i)P_{1}(t,i)C(i)
+\widehat{C}^{\top}(i)P_{1}(t,i)\widehat{C}(i)+\widehat{Q}_{1}(i)\\
&-S_{1}^{\top}(t,i)\widetilde{N}_{1}^{-1}(t,i)\widehat{S}_{1}(t,i)
-\widehat{S}_{1}^{\top}(t,i)\widetilde{N}_{1}^{-1}(t,i)S_{1}(t,i)\\
&-\widehat{S}_{1}^{\top}(t,i)\widetilde{N}_{1}^{-1}(t,i)\widehat{S}_{1}(t,i)
+\sum_{j\in\mathcal{M}}\lambda_{ij}[\widehat{P}_{1}(t,j)-\widehat{P}_{1}(t,i)]\Big],\\
\widehat{P}_{1}(T,i)=&\widehat{G}_{1}(i),\quad i\in\mathcal{M},
\end{aligned}
\right.
\end{equation}
and an auxiliary BSDE:
\begin{equation}\label{follower 3}
\left\{
\begin{aligned}
d\varphi(t)
=&-\Big[\mathbb{A}^{\top}(t,\alpha(t))\varphi(t)+\widehat{\mathbb{A}}^{\top}(t,\alpha(t))\widehat{\varphi}(t)
+\mathbb{C}^{\top}(t,\alpha(t))\theta(t)+\widehat{\mathbb{C}}^{\top}(t,\alpha(t))\widehat{\theta}(t)\\
&+\mathbb{F}_{2}^{\top}(t,\alpha(t))u_{2}(t)+\widehat{\mathbb{F}}_{2}^{\top}(t,\alpha(t))\widehat{u}_{2}(t)\Big]dt
+\theta(t)dW(t)+\eta(t)\bullet dM(t),\\
\varphi(T)=&0,
\end{aligned}
\right.
\end{equation}
where, for simplicity of presentation, we denote
\begin{equation*}
\begin{aligned}
\mathbb{A}(t,i)=&A(i)-B_{1}(i)\widetilde{N}_{1}^{-1}(t,i)S_{1}(t,i),\quad
\widehat{\mathbb{A}}(t,i)=\widehat{A}(i)-B_{1}(i)\widetilde{N}_{1}^{-1}(t,i)\widehat{S}_{1}(t,i),\\
\mathbb{C}(t,i)=&C(i)-D_{1}(i)\widetilde{N}^{-1}(t,i)S_{1}(t,i),\quad
\widehat{\mathbb{C}}(t,i)=\widehat{C}(i)-D_{1}(i)\widetilde{N}^{-1}(t,i)\widehat{S}_{1}(t,i),\\
S_{2}(t,i)=&B_{2}^{\top}(i)P_{1}(t,i)+D_{2}^{\top}(i)P_{1}(t,i)C(i),\quad
\widehat{S}_{2}(t,i)=B_{2}^{\top}(i)\widehat{P}_{1}(t,i)+D_{2}^{\top}(i)P_{1}(t,i)\widehat{C}(i),\\
\mathbb{F}_{2}(t,i)=&S_{2}(t,i)-D_{2}^{\top}(i)P_{1}(t,i)D_{1}(i)\widetilde{N}_{1}^{-1}(t,i)S_{1}(t,i),\\
\widehat{\mathbb{F}}_{2}(t,i)=&\widehat{S}_{2}(t,i)
-D_{2}^{\top}(i)P_{1}(t,i)D_{1}(i)\widetilde{N}_{1}^{-1}(t,i)\widehat{S}_{1}(t,i),\quad i\in\mathcal{M}.
\end{aligned}
\end{equation*}
Further, let $\widetilde{P}_{1}(t,i)=P_{1}(t,i)+\widehat{P}_{1}(t,i)$,
$i\in\mathcal{M}$, then we have
\begin{equation}\label{follower 5}
\left\{
\begin{aligned}
\dot{\widetilde{P}}_{1}(t,i)
=&-\Big[\widetilde{P}_{1}(t,i)\widetilde{A}(i)+\widetilde{A}^{\top}(i)\widetilde{P}_{1}(t,i)
+\widetilde{C}^{\top}(i)P_{1}(t,i)\widetilde{C}(i)+\widetilde{Q}_{1}(i)\\
&-\widetilde{S}_{1}^{\top}(t,i)\widetilde{N}_{1}^{-1}(t,i)\widetilde{S}_{1}(t,i)
+\sum_{j\in\mathcal{M}}\lambda_{ij}[\widetilde{P}_{1}(t,j)-\widetilde{P}_{1}(t,i)]\Big],\\
\widetilde{P}_{1}(T,i)=&\widetilde{G}_{1}(i),\quad i\in\mathcal{M},
\end{aligned}
\right.
\end{equation}
where $\widetilde{\Lambda}\doteq\Lambda+\widehat{\Lambda}$ for $\Lambda=A,C,Q_{1},S_{1},G_{1}$;
so we can use (\ref{follower 5}) instead of (\ref{follower 2}).
Similar to \cite[Theorem 4.1]{Yong2013}, under Assumption (A1), (\ref{follower 1}) and (\ref{follower 5})
have unique solutions $P_{1}(\cdot,i)$ and $\widetilde{P}_{1}(\cdot,i)$, $i\in\mathcal{M}$, respectively,
which are positive definite. From \cite[Theorem 3.4]{NNY2020}, (\ref{follower 3}) also admits
a unique solution $(\varphi(\cdot),\theta(\cdot),\eta(\cdot))\in \mathcal{S}_{\mathcal{F}}^{2}(R^{n})
\times \mathcal{L}_{\mathcal{F}}^{2}(R^{n})\times \mathcal{K}_{\mathcal{F}}^{2}(R^{n})$.
\end{proof}
\begin{remark}
Note that $P_{1}$ and $\widetilde{P}_{1}$ do not depend on $u_{2}$, whereas
$(\varphi,\theta,\eta)$ does depend on $u_{2}$. Moreover, since (\ref{follower 3})
is a BSDE, the value $(\varphi(t),\theta(t),\eta(t))$ of $(\varphi,\theta,\eta)$
at time $t$ depends on $\{u_{2}(s):s\in [0,T]\}$. Then, $\Phi$ and hence $u_{1}^{*}$
defined by (\ref{follower optimal control}) depend on $\{u_{2}(s):s\in [0,T]\}$ as well,
which means $u_{1}^{*}$ is anticipating in nature. Thus, it is important to find a ``real"
state feedback representation for $u_{1}^{*}$ only in terms of $x$ and $\widehat{x}$.
\end{remark}
In the following theorem, based on the so-called \emph{completion of the squares method},
we verify the optimality of (\ref{follower optimal control}) and compute the minimal cost
for the follower.
\begin{lemma}\label{Theorem follower}
Let Assumption (A1) hold. For any given $u_{2}(\cdot)\in\mathcal{U}_{2}$ for the leader,
$u_{1}^{*}(\cdot)$ defined by (\ref{follower optimal control}) is indeed an optimal control
for the follower, and
\begin{equation*}
\begin{aligned}
J_{1}(u_{1}^{*}(\cdot),&u_{2}(\cdot))
=\frac{1}{2}\langle \widetilde{P}_{1}(0,i)x_{0},x_{0}\rangle+\langle\varphi(0),x_{0}\rangle\\
&+\frac{1}{2}E\bigg[\int_{0}^{T}\Big(-|\widetilde{N}_{1}^{-\frac{1}{2}}\Phi|^{2}
+\langle D_{2}^{\top}P_{1}D_{2}u_{2},u_{2}\rangle
+2\langle B_{2}^{\top}\varphi+D_{2}^{\top}\theta,u_{2}\rangle\Big)dt\bigg].
\end{aligned}
\end{equation*}
\end{lemma}
\begin{proof}
Note that $x(0)=\widehat{x}(0)=x_{0}$, then for any $u_{1}\in\mathcal{U}_{1}$, we have
\begin{equation}\label{follower proof 1}
\begin{aligned}
&J_{1}(u_{1}(\cdot),u_{2}(\cdot))\\
=&J_{1}(u_{1}(\cdot),u_{2}(\cdot))
-\frac{1}{2}\langle P_{1}(0,i)(x(0)-\widehat{x}(0)),x(0)-\widehat{x}(0)\rangle\\
&-\frac{1}{2}\langle \widetilde{P}_{1}(0,i)\widehat{x}(0),\widehat{x}(0)\rangle
+\frac{1}{2}\langle \widetilde{P}_{1}(0,i)\widehat{x}(0),\widehat{x}(0)\rangle
-\langle\varphi(0),x(0)\rangle+\langle\varphi(0),x(0)\rangle\\
=&J_{1}(u_{1}(\cdot),u_{2}(\cdot))
+\frac{1}{2}\langle \widetilde{P}_{1}(0,i)x_{0},x_{0}\rangle
+\langle\varphi(0),x_{0}\rangle\\
&-\frac{1}{2}E\bigg[\langle P_{1}(T,\alpha(T))(x(T)-\widehat{x}(T)),x(T)-\widehat{x}(T)\rangle
-\int_{0}^{T}d\langle P_{1}(x-\widehat{x}),x-\widehat{x}\rangle\bigg]\\
&-\frac{1}{2}E\bigg[\langle \widetilde{P}_{1}(T,\alpha(T))\widehat{x}(T),\widehat{x}(T)\rangle
-\int_{0}^{T}d\langle \widetilde{P}_{1}\widehat{x},\widehat{x}\rangle\bigg]
-E\bigg[\langle\varphi(T),x(T)\rangle-\int_{0}^{T}d\langle\varphi,x\rangle\bigg]\\
=&\frac{1}{2}\langle \widetilde{P}_{1}(0,i)x_{0},x_{0}\rangle+\langle\varphi(0),x_{0}\rangle\\
&+\frac{1}{2}E\bigg[\int_{0}^{T}\Big(\langle Q_{1}(x-\widehat{x}),x-\widehat{x}\rangle
+\langle \widetilde{Q}_{1}\widehat{x},\widehat{x}\rangle
+\langle N_{1}u_{1},u_{1}\rangle\Big)dt\bigg]\\
&+\frac{1}{2}E\bigg[\int_{0}^{T}\Big(d\langle P_{1}(x-\widehat{x}),x-\widehat{x}\rangle
+d\langle \widetilde{P}_{1}\widehat{x},\widehat{x}\rangle
+2d\langle\varphi,x\rangle\Big)\bigg].
\end{aligned}
\end{equation}
On the one hand, applying It\^{o}'s formula for Markov modulated processes to $P_{1}(x-\widehat{x})$,
\begin{equation}\label{follower proof +1}
\begin{aligned}
d[P_{1}(x-\widehat{x})]
=&-[A^{\top}P_{1}+C^{\top}P_{1}C+Q_{1}-S_{1}^{\top}\widetilde{N}_{1}^{-1}S_{1}](x-\widehat{x})dt\\
&+P_{1}[B_{1}u_{1}-B_{1}\widehat{u}_{1}+B_{2}u_{2}-B_{2}\widehat{u}_{2}]dt\\
&+P_{1}[C(x-\widehat{x})+\widetilde{C}\widehat{x}+D_{1}u_{1}+D_{2}u_{2}]dW\\
&+\sum_{i,j\in\mathcal{M}}[P_{1}(t,j)-P_{1}(t,i)](x-\widehat{x})dM_{ij}.
\end{aligned}
\end{equation}
Applying It\^{o}'s formula for semi-martingales
(see Karatzas and Shreve \cite[Theorem 3.3]{KS}) to
$\langle P_{1}(x-\widehat{x}),x-\widehat{x}\rangle$
(only the $dt$ part is preserved),
\begin{equation}\label{follower proof 2}
\begin{aligned}
&d\langle P_{1}(x-\widehat{x}),x-\widehat{x}\rangle\\
=&\langle d[P_{1}(x-\widehat{x})],x-\widehat{x}\rangle
+\langle P_{1}(x-\widehat{x}),d(x-\widehat{x})\rangle
+\langle d[P_{1}(x-\widehat{x})],d(x-\widehat{x})\rangle\\
=&\langle [-C^{\top}P_{1}C-Q_{1}+S_{1}^{\top}\widetilde{N}_{1}^{-1}S_{1}](x-\widehat{x})
+P_{1}[B_{1}u_{1}-B_{1}\widehat{u}_{1}+B_{2}u_{2}-B_{2}\widehat{u}_{2}],x-\widehat{x}\rangle dt\\
&+\langle P_{1}(x-\widehat{x}),B_{1}u_{1}-B_{1}\widehat{u}_{1}+B_{2}u_{2}-B_{2}\widehat{u}_{2}\rangle dt\\
&+\langle P_{1}[C(x-\widehat{x})+\widetilde{C}\widehat{x}+D_{1}u_{1}+D_{2}u_{2}],
C(x-\widehat{x})+\widetilde{C}\widehat{x}+D_{1}u_{1}+D_{2}u_{2}\rangle dt.
\end{aligned}
\end{equation}
On the other hand,
applying It\^{o}'s formula for Markov modulated processes to $\widetilde{P}_{1}\widehat{x}$,
\begin{equation}\label{follower proof +2}
\begin{aligned}
d[\widetilde{P}_{1}\widehat{x}]
=&-[\widetilde{A}^{\top}\widetilde{P}_{1}+\widetilde{C}^{\top}P_{1}\widetilde{C}+\widetilde{Q}_{1}
-\widetilde{S}_{1}^{\top}\widetilde{N}_{1}^{-1}\widetilde{S}_{1}]\widehat{x}dt
+\widetilde{P}_{1}[B_{1}\widehat{u}_{1}+B_{2}\widehat{u}_{2}]dt\\
&+\sum_{i,j\in\mathcal{M}}[\widetilde{P}_{1}(t,j)-\widetilde{P}_{1}(t,i)]\widehat{x}dM_{ij}.
\end{aligned}
\end{equation}
Applying It\^{o}'s formula for semi-martingales to
$\langle \widetilde{P}_{1}\widehat{x},\widehat{x}\rangle$,
\begin{equation}\label{follower proof 3}
\begin{aligned}
d\langle \widetilde{P}_{1}\widehat{x},\widehat{x}\rangle
=&\langle d(\widetilde{P}_{1}\widehat{x}),\widehat{x}\rangle
+\langle \widetilde{P}_{1}\widehat{x},d\widehat{x}\rangle
+\langle d(\widetilde{P}_{1}\widehat{x}),d\widehat{x}\rangle\\
=&\langle [-\widetilde{C}^{\top}P_{1}\widetilde{C}-\widetilde{Q}_{1}
+\widetilde{S}_{1}^{\top}\widetilde{N}_{1}^{-1}\widetilde{S}_{1}]\widehat{x}
+\widetilde{P}_{1}[B_{1}\widehat{u}_{1}+B_{2}\widehat{u}_{2}],\widehat{x}\rangle dt\\
&+\langle\widetilde{P}_{1}\widehat{x},B_{1}\widehat{u}_{1}+B_{2}\widehat{u}_{2}\rangle dt.
\end{aligned}
\end{equation}
Finally, applying It\^{o}'s formula for semi-martingales to $2\langle\varphi,x\rangle$,
\begin{equation}\label{follower proof 4}
\begin{aligned}
2d\langle\varphi,x\rangle=&2(\langle d\varphi,x\rangle
+\langle\varphi,dx\rangle+\langle d\varphi,dx\rangle)\\
=&2\langle -[\mathbb{A}^{\top}\varphi+\widehat{\mathbb{A}}^{\top}\widehat{\varphi}
+\mathbb{C}^{\top}\theta+\widehat{\mathbb{C}}^{\top}\widehat{\theta}
+\mathbb{F}_{2}^{\top}u_{2}+\widehat{\mathbb{F}}_{2}^{\top}\widehat{u}_{2}],x\rangle dt\\
&+2\langle\varphi,Ax+\widehat{A}\widehat{x}+B_{1}u_{1}+B_{2}u_{2}\rangle dt\\
&+2\langle\theta,Cx+\widehat{C}\widehat{x}+D_{1}u_{1}+D_{2}u_{2}\rangle dt.
\end{aligned}
\end{equation}
We first look at the terms involving $u_{1}$ and $\widehat{u}_{1}$
in (\ref{follower proof 1})--(\ref{follower proof 4}):
\begin{equation*}
\begin{aligned}
&u_{1}^{\top}(N_{1}+D_{1}^{\top}P_{1}D_{1})u_{1}\\
&+2u_{1}^{\top}[B_{1}^{\top}P_{1}(x-\widehat{x})
+D_{1}^{\top}P_{1}(C(x-\widehat{x})+\widetilde{C}\widehat{x}+D_{2}u_{2})
+B_{1}\widetilde{P}_{1}\widehat{x}+B_{1}^{\top}\varphi+D_{1}^{\top}\theta]\\
=&|\widetilde{N}_{1}^{\frac{1}{2}}u_{1}+\widetilde{N}_{1}^{-\frac{1}{2}}[S_{1}(x-\widehat{x})+\widetilde{S}_{1}\widehat{x}+\Phi]|^{2}
-|\widetilde{N}_{1}^{-\frac{1}{2}}[S_{1}(x-\widehat{x})+\widetilde{S}_{1}\widehat{x}+\Phi]|^{2},
\end{aligned}
\end{equation*}
in which we have used Lemma \ref{Lemma conditional E} to get
\begin{equation*}
\begin{aligned}
&E\langle P_{1}B_{1}\widehat{u}_{1},x-\widehat{x}\rangle
=E\langle P_{1}B_{1}u_{1},\widehat{x}-\widehat{x}\rangle=0,\\
&E\langle \widetilde{P}_{1}B_{1}\widehat{u}_{1},\widehat{x}\rangle
=E\langle \widetilde{P}_{1}B_{1}u_{1},\widehat{x}\rangle.
\end{aligned}
\end{equation*}
For the terms involving no $u_{1}$ or $\widehat{u}_{1}$
in (\ref{follower proof 1})--(\ref{follower proof 4}):
\begin{equation*}
\begin{aligned}
&\langle S_{1}^{\top}\widetilde{N}_{1}^{-1}S_{1}(x-\widehat{x}),x-\widehat{x}\rangle
+\langle D_{2}^{\top}P_{1}D_{2}u_{2},u_{2}\rangle\\
&+2\langle B_{2}^{\top}P_{1}(x-\widehat{x}),u_{2}\rangle
+2\langle D_{2}^{\top}P_{1}[C(x-\widehat{x})+\widetilde{C}\widehat{x}],u_{2}\rangle\\
&+\langle \widetilde{S}_{1}^{\top}\widetilde{N}_{1}^{-1}\widetilde{S}_{1}\widehat{x},\widehat{x}\rangle
+2\langle B_{2}^{\top}\widetilde{P}_{1}\widehat{x},u_{2}\rangle
+2\langle B_{2}^{\top}\varphi,u_{2}\rangle+2\langle D_{2}^{\top}\theta,u_{2}\rangle\\
&+2\langle B_{1}\widetilde{N}_{1}^{-1}S_{1}x,\varphi\rangle
+2\langle B_{1}\widetilde{N}_{1}^{-1}\widehat{S}_{1}\widehat{x},\varphi\rangle
+2\langle D_{1}\widetilde{N}_{1}^{-1}S_{1}x,\theta\rangle
+2\langle D_{1}\widetilde{N}_{1}^{-1}\widehat{S}_{1}\widehat{x},\theta\rangle\\
&-2\langle[S_{2}-D_{2}^{\top}P_{1}D_{1}\widetilde{N}_{1}^{-1}S_{1}]x,u_{2}\rangle
-2\langle[\widehat{S}_{2}-D_{2}^{\top}P_{1}D_{1}\widetilde{N}_{1}^{-1}\widehat{S}_{1}]\widehat{x},u_{2}\rangle\\
=&\langle D_{2}^{\top}P_{1}D_{2}u_{2},u_{2}\rangle
+2\langle B_{2}^{\top}\varphi,u_{2}\rangle+2\langle D_{2}^{\top}\theta,u_{2}\rangle\\
&+\langle S_{1}^{\top}\widetilde{N}_{1}^{-1}S_{1}(x-\widehat{x}),x-\widehat{x}\rangle
+\langle \widetilde{S}_{1}^{\top}\widetilde{N}_{1}^{-1}\widetilde{S}_{1}\widehat{x},\widehat{x}\rangle\\
&+2\langle D_{2}^{\top}P_{1}D_{1}\widetilde{N}_{1}^{-1}S_{1}x,u_{2}\rangle
+2\langle D_{2}^{\top}P_{1}D_{1}\widetilde{N}_{1}^{-1}\widehat{S}_{1}\widehat{x},u_{2}\rangle\\
&+2\langle B_{1}\widetilde{N}_{1}^{-1}S_{1}x,\varphi\rangle
+2\langle B_{1}\widetilde{N}_{1}^{-1}\widehat{S}_{1}\widehat{x},\varphi\rangle
+2\langle D_{1}\widetilde{N}_{1}^{-1}S_{1}x,\theta\rangle
+2\langle D_{1}\widetilde{N}_{1}^{-1}\widehat{S}_{1}\widehat{x},\theta\rangle\\
=&\langle D_{2}^{\top}P_{1}D_{2}u_{2},u_{2}\rangle
+2\langle B_{2}^{\top}\varphi,u_{2}\rangle+2\langle D_{2}^{\top}\theta,u_{2}\rangle\\
&+|\widetilde{N}_{1}^{-\frac{1}{2}}[S_{1}(x-\widehat{x})+\widetilde{S}_{1}\widehat{x}+\Phi]|^{2}
-|\widetilde{N}_{1}^{-\frac{1}{2}}\Phi|^{2},
\end{aligned}
\end{equation*}
in which we have also used Lemma \ref{Lemma conditional E} to get
\begin{equation*}
\begin{aligned}
&E\langle P_{1}B_{2}\widehat{u}_{2},x-\widehat{x}\rangle
=E\langle P_{1}B_{2}u_{2},\widehat{x}-\widehat{x}\rangle=0,\\
&E\langle P_{1}C(x-\widehat{x}),\widetilde{C}\widehat{x}\rangle
=E\langle P_{1}C(\widehat{x}-\widehat{x}),\widetilde{C}x\rangle=0,\\
&E\langle \widetilde{P}_{1}B_{2}\widehat{u}_{2},\widehat{x}\rangle
=E\langle \widetilde{P}_{1}B_{2}u_{2},\widehat{x}\rangle,\\
&E\langle B_{1}\widetilde{N}_{1}^{-1}\widehat{S}_{1}x,\widehat{\varphi}\rangle
=E\langle B_{1}\widetilde{N}_{1}^{-1}\widehat{S}_{1}\widehat{x},\varphi\rangle,\\
&E\langle D_{1}\widetilde{N}_{1}^{-1}\widehat{S}_{1}x,\widehat{\theta}\rangle
=E\langle D_{1}\widetilde{N}_{1}^{-1}\widehat{S}_{1}\widehat{x},\theta\rangle,\\
&E\langle[\widehat{S}_{2}-D_{2}^{\top}P_{1}D_{1}\widetilde{N}_{1}^{-1}\widehat{S}_{1}]x,\widehat{u}_{2}\rangle
=E\langle[\widehat{S}_{2}-D_{2}^{\top}P_{1}D_{1}\widetilde{N}_{1}^{-1}\widehat{S}_{1}]\widehat{x},u_{2}\rangle.
\end{aligned}
\end{equation*}
Then, (\ref{follower proof 1}) reduces to
\begin{equation*}
\begin{aligned}
J_{1}(u_{1}(\cdot),&u_{2}(\cdot))
=\frac{1}{2}\langle \widetilde{P}_{1}(0,i)x_{0},x_{0}\rangle
+\langle\varphi(0),x_{0}\rangle\\
&+\frac{1}{2}E\bigg[\int_{0}^{T}\Big(|\widetilde{N}_{1}^{\frac{1}{2}}
(u_{1}+\widetilde{N}_{1}^{-1}[S_{1}(x-\widehat{x})+\widetilde{S}_{1}\widehat{x}+\Phi])|^{2}
-|\widetilde{N}_{1}^{-\frac{1}{2}}\Phi|^{2}\\
&+\langle D_{2}^{\top}P_{1}D_{2}u_{2},u_{2}\rangle
+2\langle B_{2}^{\top}\varphi+D_{2}^{\top}\theta,u_{2}\rangle\Big)dt\bigg].
\end{aligned}
\end{equation*}
It follows that $u_{1}^{*}$ defined by (\ref{follower optimal control})
is indeed an optimal control for the follower, and
\begin{equation*}
\begin{aligned}
J_{1}(u_{1}^{*}(\cdot),&u_{2}(\cdot))
=\frac{1}{2}\langle \widetilde{P}_{1}(0,i)x_{0},x_{0}\rangle
+\langle\varphi(0),x_{0}\rangle\\
&+\frac{1}{2}E\bigg[\int_{0}^{T}\Big(-|\widetilde{N}_{1}^{-\frac{1}{2}}\Phi|^{2}
+\langle D_{2}^{\top}P_{1}D_{2}u_{2},u_{2}\rangle
+2\langle B_{2}^{\top}\varphi+D_{2}^{\top}\theta,u_{2}\rangle\Big)dt\bigg].
\end{aligned}
\end{equation*}
The proof is completed.
\end{proof}

\section{The problem for the leader}\label{leader problem}

After the follower's problem being solved and the follower taking his optimal control
(\ref{follower optimal control}), the leader faces a state equation, which is a
conditional mean-field FBSDE with regime switching, consisting of the state equation
(\ref{system}) of the LQ problem and the auxiliary BSDE (\ref{follower 3}) of the follower:
\begin{equation}\label{leader system}
\left\{
\begin{aligned}
dx=&\Big[\mathbb{A}x+\widehat{\mathbb{A}}\widehat{x}
+\mathbb{F}_{1}\varphi+\mathbb{B}_{1}\theta+\mathbb{B}_{2}u_{2}\Big]dt
+\Big[\mathbb{C}x+\widehat{\mathbb{C}}\widehat{x}
+\mathbb{B}_{1}^{\top}\varphi+\mathbb{D}_{1}\theta+\mathbb{D}_{2}u_{2}\Big]dW,\\
d\varphi=&-\Big[\mathbb{A}^{\top}\varphi+\widehat{\mathbb{A}}^{\top}\widehat{\varphi}
+\mathbb{C}^{\top}\theta+\widehat{\mathbb{C}}^{\top}\widehat{\theta}
+\mathbb{F}_{2}^{\top}u_{2}+\widehat{\mathbb{F}}_{2}^{\top}\widehat{u}_{2}\Big]dt
+\theta dW+\eta\bullet dM,\\
x(0)=&x_{0},\quad \varphi(T)=0,
\end{aligned}
\right.
\end{equation}
where, for convenience, we denote
\begin{equation*}
\begin{aligned}
\mathbb{B}_{1}(t,i)=&-B_{1}(i)\widetilde{N}_{1}^{-1}(t,i)D_{1}^{\top}(i),\quad
\mathbb{B}_{2}(t,i)=B_{2}(i)-B_{1}(i)\widetilde{N}_{1}^{-1}(t,i)D_{1}^{\top}(i)P_{1}(t,i)D_{2}(i),\\
\mathbb{D}_{1}(t,i)=&-D_{1}(i)\widetilde{N}_{1}^{-1}(t,i)D_{1}^{\top}(i),\quad
\mathbb{D}_{2}(t,i)=D_{2}(i)-D_{1}(i)\widetilde{N}_{1}^{-1}(t,i)D_{1}^{\top}(i)P_{1}(t,i)D_{2}(i),\\
\mathbb{F}_{1}(t,i)=&-B_{1}(i)\widetilde{N}_{1}^{-1}(t,i)B_{1}^{\top}(i),\quad i\in\mathcal{M}.
\end{aligned}
\end{equation*}
Note that the FBSDE (\ref{leader system}) is \emph{decoupled} in the sense that
one can first solve the backward equation for $(\varphi,\theta,\eta)$
and then solve the forward equation for $x$,
so the unique solvability of (\ref{leader system}) is guaranteed.
The leader's problem is to find an optimal control $u_{2}^{*}(\cdot)\in\mathcal{U}_{2}$
to minimize her cost functional (\ref{cost functionals}) for $k=2$.
We will also utilize the SMP approach to solve the leader's problem.
In addition to Assumption (A1), we further make the following assumption:

(A2) $Q_{2}(i)\geq0$, $\widehat{Q}_{2}(i)\geq0$, $N_{2}(i)>0$,
$G_{2}(i)\geq0$, $\widehat{G}_{2}(i)\geq0$, $i\in\mathcal{M}$.

The adjoint equation for the leader is given by
\begin{equation}\label{leader adjoint equation}
\left\{
\begin{aligned}
dy=&-\Big[\mathbb{A}^{\top}y+\widehat{\mathbb{A}}^{\top}\widehat{y}
+\mathbb{C}^{\top}z+\widehat{\mathbb{C}}^{\top}\widehat{z}
+Q_{2}x^{*}+\widehat{Q}_{2}\widehat{x}^{*}\Big]dt+zdW+k\bullet dM,\\
d\psi=&\Big[\mathbb{A}\psi+\widehat{\mathbb{A}}\widehat{\psi}
+\mathbb{F}_{1}y+\mathbb{B}_{1}z\Big]dt
+\Big[\mathbb{C}\psi+\widehat{\mathbb{C}}\widehat{\psi}
+\mathbb{B}_{1}^{\top}y+\mathbb{D}_{1}z\Big]dW,\\
y(T)=&G_{2}(\alpha(T))x^{*}(T)+\widehat{G}_{2}(\alpha(T))\widehat{x}^{*}(T),\quad \psi(0)=0,
\end{aligned}
\right.
\end{equation}
where $(x^{*},\varphi^{*},\theta^{*},\eta^{*})$ is the corresponding solution of
(\ref{leader system}) under an optimal control $u_{2}^{*}$ for the leader.
Note that (\ref{leader adjoint equation}) is also a decoupled conditional mean-field
FBSDE with regime switching, and thereby its unique solvability is guaranteed.
Based on Yong \cite[Theorem 3.2]{Yong2002} and Nguyen et al. \cite[Theorem 3.7]{NNY2020},
one can establish the following SMP for the leader's problem.
\begin{lemma}\label{leader theorem 1}
Let Assumptions (A1) and (A2) hold. Then $u_{2}^{*}\in\mathcal{U}_{2}$ is an optimal control
for the leader if and only if the adjoint equation (\ref{leader adjoint equation}) admits
a unique solution $(y,z,k,\psi)\in\mathcal{S}_{\mathcal{F}}^{2}(R^{n})
\times\mathcal{L}_{\mathcal{F}}^{2}(R^{n})\times\mathcal{K}_{\mathcal{F}}^{2}(R^{n})
\times\mathcal{L}_{\mathcal{F}}^{2}(R^{n})$ such that
\begin{equation}\label{leader original optimal condition}
\begin{aligned}
N_{2}u_{2}^{*}+\mathbb{B}_{2}^{\top}y+\mathbb{D}_{2}^{\top}z+\mathbb{F}_{2}\psi
+\widehat{\mathbb{F}}_{2}\widehat{\psi}=0.
\end{aligned}
\end{equation}
\end{lemma}
\begin{proof}
Let $(x^{*},\varphi^{*},\theta^{*},\eta^{*})\in
\mathcal{L}_{\mathcal{F}}^{2}(R^{n})\times\mathcal{S}_{\mathcal{F}}^{2}(R^{n})
\times\mathcal{L}_{\mathcal{F}}^{2}(R^{n})\times\mathcal{K}_{\mathcal{F}}^{2}(R^{n})$
be the corresponding solution of (\ref{leader system}) under $u_{2}^{*}$.
For any $u_{2}^{0}\in\mathcal{U}_{2}$, we introduce the following state equation:
\begin{equation}\label{leader proof 1-1}
\left\{
\begin{aligned}
dx^{0}=&\Big[\mathbb{A}x^{0}+\widehat{\mathbb{A}}\widehat{x}^{0}
+\mathbb{F}_{1}\varphi^{0}+\mathbb{B}_{1}\theta^{0}+\mathbb{B}_{2}u_{2}^{0}\Big]dt\\
&+\Big[\mathbb{C}x^{0}+\widehat{\mathbb{C}}\widehat{x}^{0}
+\mathbb{B}_{1}^{\top}\varphi^{0}+\mathbb{D}_{1}\theta^{0}+\mathbb{D}_{2}u_{2}^{0}\Big]dW,\\
d\varphi^{0}=&-\Big[\mathbb{A}^{\top}\varphi^{0}+\widehat{\mathbb{A}}^{\top}\widehat{\varphi}^{0}
+\mathbb{C}^{\top}\theta^{0}+\widehat{\mathbb{C}}^{\top}\widehat{\theta}^{0}
+\mathbb{F}_{2}^{\top}u_{2}^{0}+\widehat{\mathbb{F}}_{2}^{\top}\widehat{u}_{2}^{0}\Big]dt\\
&+\theta^{0}dW+\eta^{0}\bullet dM,\\
x^{0}(0)=&0,\quad \varphi^{0}(T)=0,
\end{aligned}
\right.
\end{equation}
and the adjoint equation:
\begin{equation}\label{leader proof 1-2}
\left\{
\begin{aligned}
dy^{0}=&-\Big[\mathbb{A}^{\top}y^{0}+\widehat{\mathbb{A}}^{\top}\widehat{y}^{0}
+\mathbb{C}^{\top}z^{0}+\widehat{\mathbb{C}}^{\top}\widehat{z}^{0}
+Q_{2}x^{0}+\widehat{Q}_{2}\widehat{x}^{0}\Big]dt\\
&+z^{0}dW+k^{0}\bullet dM,\\
d\psi^{0}=&\Big[\mathbb{A}\psi^{0}+\widehat{\mathbb{A}}\widehat{\psi}^{0}
+\mathbb{F}_{1}y^{0}+\mathbb{B}_{1}z^{0}\Big]dt\\
&+\Big[\mathbb{C}\psi^{0}+\widehat{\mathbb{C}}\widehat{\psi}^{0}
+\mathbb{B}_{1}^{\top}y^{0}+\mathbb{D}_{1}z^{0}\Big]dW,\\
y^{0}(T)=&G_{2}(\alpha(T))x^{0}(T)+\widehat{G}_{2}(\alpha(T))\widehat{x}^{0}(T),\quad \psi^{0}(0)=0.
\end{aligned}
\right.
\end{equation}
Note that the initial condition $x^{0}(0)=0$ in (\ref{leader proof 1-1}),
which is the only difference compared with (\ref{leader system}).
Also, the FBSDEs (\ref{leader proof 1-1}) and (\ref{leader proof 1-2})
have a unique solution
$(x^{0},\varphi^{0},\theta^{0},\eta^{0},y^{0},z^{0},k^{0},\psi^{0})$
in the usual space.
		
For any $\lambda\in R$, consider $u_{2}\doteq u_{2}^{*}+\lambda u_{2}^{0}\in\mathcal{U}_{2}$
and denote $(x,\varphi,\theta,\eta)$ the corresponding solution of (\ref{leader system}).
From the linearity of the above FBSDEs, we have $x=x^{*}+\lambda x^{0}$. Then,
\begin{equation}\label{leader proof 2}
\begin{aligned}
&J_{2}(u_{1}^{*},u_{2})-J_{2}(u_{1}^{*},u_{2}^{*})\\
=&\frac{\lambda^{2}}{2}E\bigg[\int_{0}^{T}\Big(\langle Q_{2}x^{0},x^{0}\rangle
+\langle \widehat{Q}_{2}\widehat{x}^{0},\widehat{x}^{0}\rangle
+\langle N_{2}u_{2}^{0},u_{2}^{0}\rangle\Big)dt\\
&+\langle G_{2}(\alpha(T))x^{0}(T),x^{0}(T)\rangle
+\langle \widehat{G}_{2}(\alpha(T))\widehat{x}^{0}(T),\widehat{x}^{0}(T)\rangle\bigg]\\
&+\lambda E\bigg[\int_{0}^{T}\Big(\langle Q_{2}x^{*},x^{0}\rangle
+\langle \widehat{Q}_{2}\widehat{x}^{*},\widehat{x}^{0}\rangle
+\langle N_{2}u_{2}^{*},u_{2}^{0}\rangle\Big)dt\\
&+\langle G_{2}(\alpha(T))x^{*}(T),x^{0}(T)\rangle
+\langle \widehat{G}_{2}(\alpha(T))\widehat{x}^{*}(T),\widehat{x}^{0}(T)\rangle\bigg]\\
=&\frac{\lambda^{2}}{2}E\bigg[\int_{0}^{T}\Big(\langle Q_{2}x^{0},x^{0}\rangle
+\langle \widehat{Q}_{2}\widehat{x}^{0},\widehat{x}^{0}\rangle
+\langle N_{2}u_{2}^{0},u_{2}^{0}\rangle\Big)dt\\
&+\langle G_{2}(\alpha(T))x^{0}(T)+\widehat{G}_{2}(\alpha(T))\widehat{x}^{0}(T),x^{0}(T)\rangle\bigg]\\
&+\lambda E\bigg[\int_{0}^{T}\Big(\langle Q_{2}x^{*},x^{0}\rangle
+\langle \widehat{Q}_{2}\widehat{x}^{*},\widehat{x}^{0}\rangle
+\langle N_{2}u_{2}^{*},u_{2}^{0}\rangle\Big)dt\\
&+\langle G_{2}(\alpha(T))x^{*}(T)+\widehat{G}_{2}(\alpha(T))\widehat{x}^{*}(T),x^{0}(T)\rangle\bigg].
\end{aligned}
\end{equation}
On the one hand,
\begin{equation}\label{leader proof 3}
\begin{aligned}
&E[\langle G_{2}(\alpha(T))x^{0}(T)+\widehat{G}_{2}(\alpha(T))\widehat{x}^{0}(T),x^{0}(T)\rangle]\\
=&E[\langle y^{0}(T),x^{0}(T)\rangle]\\
=&E[\langle y^{0}(T),x^{0}(T)\rangle-\langle y^{0}(0),x^{0}(0)\rangle
-\langle \psi^{0}(T),\varphi^{0}(T)\rangle+\langle \psi^{0}(0),\varphi^{0}(0)\rangle]\\
=&E\bigg[\int_{0}^{T}\Big(-\langle Q_{2}x^{0},x^{0}\rangle
-\langle \widehat{Q}_{2}\widehat{x}^{0},x^{0}\rangle
+\langle u_{2}^{0},\mathbb{B}_{2}^{\top}y^{0}+\mathbb{D}_{2}^{\top}z^{0}
+\mathbb{F}_{2}\psi^{0}+\widehat{\mathbb{F}}_{2}\widehat{\psi}^{0}\rangle\Big)dt\bigg].
\end{aligned}
\end{equation}
Therefore,
\begin{equation}\label{leader proof 4}
\begin{aligned}
&E\bigg[\int_{0}^{T}\langle u_{2}^{0},N_{2}u_{2}^{0}+\mathbb{B}_{2}^{\top}y^{0}+\mathbb{D}_{2}^{\top}z^{0}
+\mathbb{F}_{2}\psi^{0}+\widehat{\mathbb{F}}_{2}\widehat{\psi}^{0}\rangle dt\bigg]\\
=&E\bigg[\int_{0}^{T}\Big(\langle Q_{2}x^{0},x^{0}\rangle
+\langle \widehat{Q}_{2}\widehat{x}^{0},x^{0}\rangle
+\langle N_{2}u_{2}^{0},u_{2}^{0}\rangle\Big)dt\\
&+\langle G_{2}(\alpha(T))x^{0}(T)+\widehat{G}_{2}(\alpha(T))\widehat{x}^{0}(T),x^{0}(T)\rangle\bigg]\\
=&E\bigg[\int_{0}^{T}\Big(\langle Q_{2}x^{0},x^{0}\rangle
+\langle \widehat{Q}_{2}\widehat{x}^{0},\widehat{x}^{0}\rangle
+\langle N_{2}u_{2}^{0},u_{2}^{0}\rangle\Big)dt\\
&+\langle G_{2}(\alpha(T))x^{0}(T),x^{0}(T)\rangle
+\langle \widehat{G}_{2}(\alpha(T))\widehat{x}^{0}(T),\widehat{x}^{0}(T)\rangle\bigg]\geq0,
\end{aligned}
\end{equation}
where we have used Assumption (A2)
and the following facts (noting Lemma \ref{Lemma conditional E}):
\begin{equation*}
\begin{aligned}
&E\langle\widehat{Q}_{2}\widehat{x}^{0},x^{0}\rangle
=E\langle\widehat{Q}_{2}\widehat{x}^{0},\widehat{x}^{0}\rangle\geq0,\\
&E\langle\widehat{G}_{2}\widehat{x}^{0}(T),x^{0}(T)\rangle
=E\langle\widehat{G}_{2}\widehat{x}^{0}(T),\widehat{x}^{0}(T)\rangle\geq0.
\end{aligned}
\end{equation*}
On the other hand,
\begin{equation}\label{leader proof 5}
\begin{aligned}
&E[\langle G_{2}(\alpha(T))x^{*}(T)+\widehat{G}_{2}(\alpha(T))\widehat{x}^{*}(T),x^{0}(T)\rangle]\\
=&E[\langle y(T),x^{0}(T)\rangle]\\
=&E[\langle y(T),x^{0}(T)\rangle-\langle y(0),x^{0}(0)\rangle
-\langle \psi(T),\varphi^{0}(T)\rangle+\langle \psi(0),\varphi^{0}(0)\rangle]\\
=&E\bigg[\int_{0}^{T}\Big(-\langle Q_{2}x^{*},x^{0}\rangle
-\langle \widehat{Q}_{2}\widehat{x}^{*},x^{0}\rangle
+\langle u_{2}^{0},\mathbb{B}_{2}^{\top}y+\mathbb{D}_{2}^{\top}z
+\mathbb{F}_{2}\psi+\widehat{\mathbb{F}}_{2}\widehat{\psi}\rangle\Big)dt\bigg].
\end{aligned}
\end{equation}
Thus, combining (\ref{leader proof 2}), (\ref{leader proof 3}), and (\ref{leader proof 5})
leads to
\begin{equation*}
\begin{aligned}
&J_{2}(u_{1}^{*},u_{2})-J_{2}(u_{1}^{*},u_{2}^{*})\\
=&\frac{\lambda^{2}}{2}
E\bigg[\int_{0}^{T}\langle u_{2}^{0},N_{2}u_{2}^{0}+\mathbb{B}_{2}^{\top}y^{0}+\mathbb{D}_{2}^{\top}z^{0}
+\mathbb{F}_{2}\psi^{0}+\widehat{\mathbb{F}}_{2}\widehat{\psi}^{0}\rangle dt\bigg]\\
&+\lambda E\bigg[\int_{0}^{T}\langle u_{2}^{0},N_{2}u_{2}^{*}
+\mathbb{B}_{2}^{\top}y+\mathbb{D}_{2}^{\top}z
+\mathbb{F}_{2}\psi+\widehat{\mathbb{F}}_{2}\widehat{\psi}\rangle dt\bigg].
\end{aligned}
\end{equation*}
From (\ref{leader proof 4}), we deduce that $u_{2}^{*}$ is optimal
if and only if
\begin{equation*}
\begin{aligned}
N_{2}u_{2}^{*}+\mathbb{B}_{2}^{\top}y+\mathbb{D}_{2}^{\top}z
+\mathbb{F}_{2}\psi+\widehat{\mathbb{F}}_{2}\widehat{\psi}=0.
\end{aligned}
\end{equation*}
The proof is completed.
\end{proof}
Similar to the follower's problem, we also expect to derive a state feedback representation
for $u_{2}^{*}$ defined by (\ref{leader original optimal condition}), which, as shown later,
is non-anticipating. To apply the dimensional augmentation approach by Yong \cite{Yong2002},
we denote
\begin{equation*}
\begin{aligned}
			X=\left[
			\begin{array}{c}
				x^{*} \\
				\psi \\
			\end{array}
			\right],\quad
			Y=\left[
			\begin{array}{c}
				y \\
				\varphi^{*} \\
			\end{array}
			\right],\quad
			Z=\left[
			\begin{array}{c}
				z \\
				\theta^{*} \\
			\end{array}
			\right],\quad
			K=\left[
			\begin{array}{c}
				k \\
				\eta^{*} \\
			\end{array}
			\right],\quad
			X_{0}=\left[
			\begin{array}{c}
				x_{0} \\
				0 \\
			\end{array}
			\right],
\end{aligned}
\end{equation*}
\begin{equation*}
\begin{aligned}
			\mathbf{A}=\left[
			\begin{array}{cc}
				\mathbb{A} & 0 \\
				0 & \mathbb{A} \\
			\end{array}
			\right],\quad
			\widehat{\mathbf{A}}=\left[
			\begin{array}{cc}
				\widehat{\mathbb{A}} & 0 \\
				0 & \widehat{\mathbb{A}} \\
			\end{array}
			\right],\quad
			\mathbf{C}=\left[
			\begin{array}{cc}
				\mathbb{C} & 0 \\
				0 & \mathbb{C} \\
			\end{array}
			\right],\quad
			\widehat{\mathbf{C}}=\left[
			\begin{array}{cc}
				\widehat{\mathbb{C}} & 0 \\
				0 & \widehat{\mathbb{C}} \\
			\end{array}
			\right],
\end{aligned}
\end{equation*}
\begin{equation*}
\begin{aligned}
			\mathbf{B}_{1}=\left[
			\begin{array}{cc}
				0 & \mathbb{B}_{1} \\
				\mathbb{B}_{1} & 0 \\
			\end{array}
			\right],\quad
			\mathbf{B}_{2}=\left[
			\begin{array}{c}
				\mathbb{B}_{2} \\
				0 \\
			\end{array}
			\right],\quad
			\mathbf{D}_{1}=\left[
			\begin{array}{cc}
				0 & \mathbb{D}_{1} \\
				\mathbb{D}_{1} & 0 \\
			\end{array}
			\right],\quad
			\mathbf{D}_{2}=\left[
			\begin{array}{c}
				\mathbb{D}_{2} \\
				0 \\
			\end{array}
			\right],
\end{aligned}
\end{equation*}
\begin{equation*}
\begin{aligned}
			\mathbf{F}_{1}=\left[
			\begin{array}{cc}
				0 & \mathbb{F}_{1} \\
				\mathbb{F}_{1} & 0 \\
			\end{array}
			\right],\quad
			\mathbf{F}_{2}=\left[
			\begin{array}{cc}
				0 & \mathbb{F}_{2} \\
			\end{array}
			\right],\quad
			\widehat{\mathbf{F}}_{2}=\left[
			\begin{array}{cc}
				0 & \widehat{\mathbb{F}}_{2} \\
			\end{array}
			\right],
\end{aligned}
\end{equation*}
\begin{equation*}
\begin{aligned}
			\mathbf{Q}_{2}=\left[
			\begin{array}{cc}
				Q_{2} & 0 \\
				0 & 0 \\
			\end{array}
			\right],\quad
			\widehat{\mathbf{Q}}_{2}=\left[
			\begin{array}{cc}
				\widehat{Q}_{2} & 0 \\
				0 & 0 \\
			\end{array}
			\right],\quad
			\mathbf{G}_{2}=\left[
			\begin{array}{cc}
				G_{2} & 0 \\
				0 & 0 \\
			\end{array}
			\right],\quad
			\widehat{\mathbf{G}}_{2}=\left[
			\begin{array}{cc}
				\widehat{G}_{2} & 0 \\
				0 & 0 \\
			\end{array}
			\right].
\end{aligned}
\end{equation*}
Then, (\ref{leader system}) and (\ref{leader adjoint equation})
can be rewritten as
\begin{equation}\label{leader Hamiltonian system}
\left\{
\begin{aligned}
dX=&\Big[\mathbf{A}X+\widehat{\mathbf{A}}\widehat{X}
+\mathbf{F}_{1}Y+\mathbf{B}_{1}Z+\mathbf{B}_{2}u_{2}^{*}\Big]dt\\
&+\Big[\mathbf{C}X+\widehat{\mathbf{C}}\widehat{X}
+\mathbf{B}_{1}^{\top}Y+\mathbf{D}_{1}Z+\mathbf{D}_{2}u_{2}^{*}\Big]dW,\\
dY=&-\Big[\mathbf{A}^{\top}Y+\widehat{\mathbf{A}}^{\top}\widehat{Y}
+\mathbf{C}^{\top}Z+\widehat{\mathbf{C}}^{\top}\widehat{Z}
+\mathbf{Q}_{2}X+\widehat{\mathbf{Q}}_{2}\widehat{X}\\
&+\mathbf{F}_{2}^{\top}u_{2}^{*}+\widehat{\mathbf{F}}_{2}^{\top}\widehat{u}_{2}^{*}\Big]dt
+ZdW+K\bullet dM,\\
X(0)=&X_{0},\quad Y(T)=\mathbf{G}_{2}(\alpha(T))X(T)
+\widehat{\mathbf{G}}_{2}(\alpha(T))\widehat{X}(T),
\end{aligned}
\right.
\end{equation}
and (\ref{leader original optimal condition}) becomes
\begin{equation}\label{leader optimal condition}
\begin{aligned}
0=N_{2}u_{2}^{*}+\mathbf{B}_{2}^{\top}Y
+\mathbf{D}_{2}^{\top}Z+\mathbf{F}_{2}X
+\widehat{\mathbf{F}}_{2}\widehat{X}.
\end{aligned}
\end{equation}
\begin{theorem}
Let Assumptions (A1) and (A2) hold. An optimal control $u_{2}^{*}$ for the leader
is given by
\begin{equation}\label{leader optimal control}
\begin{aligned}
u_{2}^{*}(t)=-\widetilde{N}_{2}^{-1}(t,\alpha(t))\Big[\mathbf{S}_{2}(t,\alpha(t))X(t)
+\widehat{\mathbf{S}}_{2}(t,\alpha(t))\widehat{X}(t)\Big],
\end{aligned}
\end{equation}
where, for the sake of simplicity, we denote
\begin{equation*}
\begin{aligned}
\widetilde{N}_{2}(t,i)=&N_{2}(i)+\mathbf{D}_{2}^{\top}(t,i)(I-P_{2}(t,i)\mathbf{D}_{1}(t,i))^{-1}P_{2}(t,i)\mathbf{D}_{2}(t,i),\\
\mathbf{J}_{2}(t,i)=&\mathbf{B}_{1}^{\top}(t,i)P_{2}(t,i)+\mathbf{C}(t,i),\quad
\widehat{\mathbf{J}}_{2}(t,i)=\mathbf{B}_{1}^{\top}(t,i)\widehat{P}_{2}(t,i)+\widehat{\mathbf{C}}(t,i),\\
\mathbf{S}_{2}(t,i)=&\mathbf{D}_{2}^{\top}(t,i)(I-P_{2}(t,i)\mathbf{D}_{1}(t,i))^{-1}P_{2}(t,i)\mathbf{J}_{2}(t,i)
+\mathbf{B}_{2}^{\top}(t,i)P_{2}(t,i)+\mathbf{F}_{2}(t,i),\\
\widehat{\mathbf{S}}_{2}(t,i)=&\mathbf{D}_{2}^{\top}(t,i)(I-P_{2}(t,i)\mathbf{D}_{1}(t,i))^{-1}P_{2}(t,i)\widehat{\mathbf{J}}_{2}(t,i)
+\mathbf{B}_{2}^{\top}(t,i)\widehat{P}_{2}(t,i)+\widehat{\mathbf{F}}_{2}(t,i),\quad i\in\mathcal{M},
\end{aligned}
\end{equation*}
provided $\widetilde{N}_{2}$ and $(I-P_{2}\mathbf{D}_{1})$ are invertible
and $P_{2}(\cdot,i)$ and $\widehat{P}_{2}(\cdot,i)$, $i\in\mathcal{M}$,
are solutions of Riccati equations (\ref{leader 1}) and (\ref{leader 2}), respectively.
\end{theorem}
\begin{proof}
In the light of the terminal condition of (\ref{leader Hamiltonian system}),
it is natural to set
\begin{equation}\label{leader FSS}
\begin{aligned}
Y(t)=P_{2}(t,\alpha(t))X(t)+\widehat{P}_{2}(t,\alpha(t))\widehat{X}(t),
\end{aligned}
\end{equation}
for some $R^{2n\times 2n}$-valued deterministic, differentiable, and symmetric
functions $P_{2}(t,i)$ and $\widehat{P}_{2}(t,i)$, $i\in \mathcal{M}$.
Applying It\^{o}'s formula for Markov-modulated processes to (\ref{leader FSS}), we have
\begin{equation}\label{leader dY}
\begin{aligned}
&dY=\Big(\dot{P}_{2}+\sum_{j\in\mathcal{M}}\lambda_{\alpha(t),j}[P_{2}(t,j)-P_{2}(t,\alpha(t))]\Big)Xdt
+\sum_{i,j\in\mathcal{M}}[P_{2}(t,j)-P_{2}(t,i)]XdM_{ij}\\
&+P_{2}\Big[\mathbf{A}X+\widehat{\mathbf{A}}\widehat{X}
+\mathbf{F}_{1}Y+\mathbf{B}_{1}Z+\mathbf{B}_{2}u_{2}^{*}\Big]dt
+P_{2}\Big[\mathbf{C}X+\widehat{\mathbf{C}}\widehat{X}
+\mathbf{B}_{1}^{\top}Y+\mathbf{D}_{1}Z+\mathbf{D}_{2}u_{2}^{*}\Big]dW\\
&+\Big(\dot{\widehat{P}}_{2}
+\sum_{j\in\mathcal{M}}\lambda_{\alpha(t),j}[\widehat{P}_{2}(t,j)-\widehat{P}_{2}(t,\alpha(t))]\Big)\widehat{X}dt
+\sum_{i,j\in\mathcal{M}}[\widehat{P}_{2}(t,j)-\widehat{P}_{2}(t,i)]\widehat{X}dM_{ij}\\
&+\widehat{P}_{2}\Big[(\mathbf{A}+\widehat{\mathbf{A}})\widehat{X}
+\mathbf{F}_{1}\widehat{Y}+\mathbf{B}_{1}\widehat{Z}+\mathbf{B}_{2}\widehat{u}_{2}^{*}\Big]dt.
\end{aligned}
\end{equation}	
Comparing the coefficients of $dW$ parts in (\ref{leader Hamiltonian system})
and (\ref{leader dY}), we obtain
\begin{equation}\label{leader Z}
\begin{aligned}
Z=(I-P_{2}\mathbf{D}_{1})^{-1}P_{2}
\Big[\mathbf{J}_{2}X+\widehat{\mathbf{J}}_{2}\widehat{X}+\mathbf{D}_{2}u_{2}^{*}\Big].
\end{aligned}
\end{equation}
Substituting (\ref{leader FSS}) and (\ref{leader Z}) into (\ref{leader optimal condition})
and observing that $(I-P_{2}\mathbf{D}_{1})^{-1}P_{2}$ is symmetric, we get
\begin{equation*}
\begin{aligned}
u_{2}^{*}=&-\widetilde{N}_{2}^{-1}\Big[\mathbf{S}_{2}X
+\widehat{\mathbf{S}}_{2}\widehat{X}\Big].
\end{aligned}
\end{equation*}
Inserting (\ref{leader FSS}), (\ref{leader Z}), and (\ref{leader optimal control})
into (\ref{leader Hamiltonian system}) and (\ref{leader dY}), respectively, we have
\begin{equation}\label{leader compare 1}
\begin{aligned}
dY=&-\Big[\Big(\mathbf{A}^{\top}P_{2}+\mathbf{Q}_{2}
+\mathbf{C}^{\top}(I-P_{2}\mathbf{D}_{1})^{-1}P_{2}\mathbf{J}_{2}\\
&-\mathbf{C}^{\top}(I-P_{2}\mathbf{D}_{1})^{-1}P_{2}\mathbf{D}_{2}\widetilde{N}_{2}^{-1}\mathbf{S}_{2}
-\mathbf{F}_{2}^{\top}\widetilde{N}_{2}^{-1}\mathbf{S}_{2}\Big)X\\
&+\Big(\mathbf{A}^{\top}\widehat{P}_{2}+\widehat{\mathbf{A}}^{\top}(P_{2}+\widehat{P}_{2})+\widehat{\mathbf{Q}}_{2}\\
&+\mathbf{C}^{\top}(I-P_{2}\mathbf{D}_{1})^{-1}P_{2}\widehat{\mathbf{J}}_{2}
+\widehat{\mathbf{C}}^{\top}(I-P_{2}\mathbf{D}_{1})^{-1}P_{2}(\mathbf{J}_{2}+\widehat{\mathbf{J}}_{2})\\
&-\mathbf{C}^{\top}(I-P_{2}\mathbf{D}_{1})^{-1}P_{2}\mathbf{D}_{2}\widetilde{N}_{2}^{-1}\widehat{\mathbf{S}}_{2}\\
&-\widehat{\mathbf{C}}^{\top}(I-P_{2}\mathbf{D}_{1})^{-1}P_{2}\mathbf{D}_{2}\widetilde{N}_{2}^{-1}
(\mathbf{S}_{2}+\widehat{\mathbf{S}}_{2})\\
&-\mathbf{F}_{2}^{\top}\widetilde{N}_{2}^{-1}\widehat{\mathbf{S}}_{2}
-\widehat{\mathbf{F}}_{2}^{\top}\widetilde{N}_{2}^{-1}(\mathbf{S}_{2}
+\widehat{\mathbf{S}}_{2})\Big)\widehat{X}\Big]dt\\
&+\Big\{\cdots\Big\}dW+\Big\{\cdots\Big\}\bullet dM,
\end{aligned}
\end{equation}
and
\begin{equation}\label{leader compare 2}
\begin{aligned}
dY=&\Big[\Big(\dot{P}_{2}+P_{2}\mathbf{A}+P_{2}\mathbf{F}_{1}P_{2}
+P_{2}\mathbf{B}_{1}(I-P_{2}\mathbf{D}_{1})^{-1}P_{2}\mathbf{J}_{2}\\
&-P_{2}\mathbf{B}_{1}(I-P_{2}\mathbf{D}_{1})^{-1}P_{2}\mathbf{D}_{2}\widetilde{N}_{2}^{-1}\mathbf{S}_{2}
-P_{2}\mathbf{B}_{2}\widetilde{N}_{2}^{-1}\mathbf{S}_{2}\\
&+\sum_{j\in\mathcal{M}}\lambda_{\alpha(t),j}[P_{2}(t,j)-P_{2}(t,\alpha(t))]\Big)X\\
&+\Big(\dot{\widehat{P}}_{2}+P_{2}\widehat{\mathbf{A}}+\widehat{P}_{2}(\mathbf{A}+\widehat{\mathbf{A}})
+P_{2}\mathbf{F}_{1}\widehat{P}_{2}+\widehat{P}_{2}\mathbf{F}_{1}(P_{2}+\widehat{P}_{2})\\
&+P_{2}\mathbf{B}_{1}(I-P_{2}\mathbf{D}_{1})^{-1}P_{2}\widehat{\mathbf{J}}_{2}
+\widehat{P}_{2}\mathbf{B}_{1}(I-P_{2}\mathbf{D}_{1})^{-1}P_{2}(\mathbf{J}_{2}+\widehat{\mathbf{J}}_{2})\\
&-P_{2}\mathbf{B}_{1}(I-P_{2}\mathbf{D}_{1})^{-1}P_{2}\mathbf{D}_{2}\widetilde{N}_{2}^{-1}\widehat{\mathbf{S}}_{2}
-P_{2}\mathbf{B}_{2}\widetilde{N}_{2}^{-1}\widehat{\mathbf{S}}_{2}\\
&-\widehat{P}_{2}\mathbf{B}_{1}(I-P_{2}\mathbf{D}_{1})^{-1}P_{2}\mathbf{D}_{2}(\mathbf{S}_{2}+\widehat{\mathbf{S}}_{2})
-\widehat{P}_{2}\mathbf{B}_{2}\widetilde{N}_{2}^{-1}(\mathbf{S}_{2}+\widehat{\mathbf{S}}_{2})\\
&+\sum_{j\in\mathcal{M}}\lambda_{\alpha(t),j}[\widehat{P}_{2}(t,j)-\widehat{P}_{2}(t,\alpha(t))]\Big)\widehat{X}\Big]dt\\
&+\Big\{\cdots\Big\}dW+\Big\{\cdots\Big\}\bullet dM.
\end{aligned}
\end{equation}
By equalizing the coefficients of $X$ and $\widehat{X}$ in
(\ref{leader compare 1}) and (\ref{leader compare 2}),
we obtain the following two Riccati equations
\begin{equation}\label{leader 1}
\left\{
\begin{aligned}
\dot{P}_{2}(t,i)
=&-\Big[P_{2}(t,i)\mathbf{A}(t,i)+\mathbf{A}^{\top}(t,i)P_{2}(t,i)
+P_{2}(t,i)\mathbf{F}_{1}(t,i)P_{2}(t,i)+\mathbf{Q}_{2}(i)\\
&+\mathbf{J}_{2}^{\top}(t,i)(I-P_{2}(t,i)\mathbf{D}_{1}(t,i))^{-1}P_{2}(t,i)\mathbf{J}_{2}(t,i)\\
&-\mathbf{S}_{2}^{\top}(t,i)\widetilde{N}_{2}^{-1}(t,i)\mathbf{S}_{2}(t,i)
+\sum_{j\in\mathcal{M}}\lambda_{ij}[P_{2}(t,j)-P_{2}(t,i)]\Big],\\
P_{2}(T,i)=&\mathbf{G}_{2}(i),\quad i\in\mathcal{M},
\end{aligned}
\right.
\end{equation}
and
\begin{equation}\label{leader 2}
\left\{
\begin{aligned}
\dot{\widehat{P}}_{2}(t,i)
=&-\Big[\widehat{P}_{2}(t,i)(\mathbf{A}(t,i)+\widehat{\mathbf{A}}(t,i))
+(\mathbf{A}(t,i)+\widehat{\mathbf{A}}(t,i))^{\top}\widehat{P}_{2}(t,i)\\
&+P_{2}(t,i)\widehat{\mathbf{A}}(t,i)+\widehat{\mathbf{A}}^{\top}(t,i)P_{2}(t,i)
+P_{2}(t,i)\mathbf{F}_{1}(t,i)\widehat{P}_{2}(t,i)\\
&+\widehat{P}_{2}(t,i)\mathbf{F}_{1}(t,i)P_{2}(t,i)
+\widehat{P}_{2}(t,i)\mathbf{F}_{1}(t,i)\widehat{P}_{2}(t,i)+\widehat{\mathbf{Q}}_{2}(i)\\
&+\mathbf{J}_{2}^{\top}(t,i)(I-P_{2}(t,i)\mathbf{D}_{1}(t,i))^{-1}P_{2}(t,i)\widehat{\mathbf{J}}_{2}(t,i)\\
&+\widehat{\mathbf{J}}_{2}^{\top}(t,i)(I-P_{2}(t,i)\mathbf{D}_{1}(t,i))^{-1}P_{2}(t,i)\mathbf{J}_{2}(t,i)\\
&+\widehat{\mathbf{J}}_{2}^{\top}(t,i)(I-P_{2}(t,i)\mathbf{D}_{1}(t,i))^{-1}P_{2}(t,i)\widehat{\mathbf{J}}_{2}(t,i)\\
&-\mathbf{S}_{2}^{\top}(t,i)\widetilde{N}_{2}^{-1}(t,i)\widehat{\mathbf{S}}_{2}(t,i)
-\widehat{\mathbf{S}}_{2}^{\top}(t,i)\widetilde{N}_{2}^{-1}(t,i)\mathbf{S}_{2}(t,i)\\
&-\widehat{\mathbf{S}}_{2}^{\top}(t,i)\widetilde{N}_{2}^{-1}(t,i)\widehat{\mathbf{S}}_{2}(t,i)
+\sum_{j\in\mathcal{M}}\lambda_{ij}[\widehat{P}_{2}(t,j)-\widehat{P}_{2}(t,i)]\Big],\\
\widehat{P}_{2}(T,i)=&\widehat{\mathbf{G}}_{2}(i),\quad i\in\mathcal{M}.
\end{aligned}
\right.
\end{equation}
As the follower's problem, we can also let $\widetilde{P}_{2}(t,i)=P_{2}(t,i)+\widehat{P}_{2}(t,i)$,
$i\in\mathcal{M}$, to get an equation that is structurally similar to (\ref{leader 1}) and can be used
instead of (\ref{leader 2}), i.e.,
\begin{equation}\label{leader 4}
\left\{
\begin{aligned}
\dot{\widetilde{P}}_{2}(t,i)
=&-\Big[\widetilde{P}_{2}(t,i)\widetilde{\mathbf{A}}(t,i)+\widetilde{\mathbf{A}}^{\top}(t,i)\widetilde{P}_{2}(t,i)
+\widetilde{P}_{2}(t,i)\mathbf{F}_{1}(t,i)\widetilde{P}_{2}(t,i)+\widetilde{\mathbf{Q}}_{2}(i)\\
&+\widetilde{\mathbf{J}}_{2}^{\top}(t,i)(I-P_{2}(t,i)\mathbf{D}_{1}(t,i))^{-1}P_{2}(t,i)\widetilde{\mathbf{J}}_{2}(t,i)\\
&-\widetilde{\mathbf{S}}_{2}^{\top}(t,i)\widetilde{N}_{2}^{-1}(t,i)\widetilde{\mathbf{S}}_{2}(t,i)
+\sum_{j\in\mathcal{M}}\lambda_{ij}[\widetilde{P}_{2}(t,j)-\widetilde{P}_{2}(t,i)]\Big],\\
\widetilde{P}_{2}(T,i)=&\widetilde{\mathbf{G}}_{2}(i),\quad i\in\mathcal{M},
\end{aligned}
\right.
\end{equation}
where $\widetilde{\mathbf{H}}\doteq\mathbf{H}+\widehat{\mathbf{H}}$
for $\mathbf{H}=\mathbf{A},\mathbf{Q}_{2},\mathbf{J}_{2},\mathbf{S}_{2},\mathbf{G}_{2}$.
\end{proof}
Then, we compute the minimal cost for the leader under $u_{2}^{*}$
defined by (\ref{leader optimal control}), and derive the \emph{non-anticipating}
state feedback representation of the follower's optimal control (\ref{follower optimal control}).
\begin{theorem}\label{leader theorem 2}
Let Assumptions (A1) and (A2) hold. Suppose that the Riccati equations (\ref{leader 1})
and (\ref{leader 4}) have solutions $P_{2}(\cdot,i)$ and $\widetilde{P}_{2}(\cdot,i)$,
$i\in\mathcal{M}$, respectively, such that $\widetilde{N}_{2}$ and $(I-P_{2}\mathbf{D}_{1})$
are invertible. Then,
\begin{equation}\label{leader optimal cost}
\begin{aligned}
J_{2}(u_{1}^{*}(\cdot),u_{2}^{*}(\cdot))
=\langle \widetilde{P}_{2}^{(11)}(0,i)x_{0},x_{0}\rangle,
\end{aligned}
\end{equation}
where $\widetilde{P}_{2}^{(11)}(0,i)$ is taken from
$$
\widetilde{P}_{2}(0,i)=\left(
\begin{array}{cc}
\widetilde{P}_{2}^{(11)}(0,i) & \widetilde{P}_{2}^{(12)}(0,i) \\
(\widetilde{P}_{2}^{(12)})^{\top}(0,i) & \widetilde{P}_{2}^{(22)}(0,i) \\
\end{array}
\right).
$$
Moreover, the non-anticipating state feedback representation of the follower's
optimal control (\ref{follower optimal control}) is given by (\ref{follower non anticipating}).
\end{theorem}
\begin{proof}
Note that
\begin{equation*}
\begin{aligned}
&E[\langle y(T),x^{*}(T)\rangle-\langle y(0),x^{*}(0)\rangle
-\langle \psi(T),\varphi^{*}(T)\rangle+\langle \psi(0),\varphi^{*}(0)\rangle]\\
=&E[\langle G_{2}(\alpha(T))x^{*}(T)+\widehat{G}_{2}(\alpha(T))\widehat{x}^{*}(T),x^{*}(T)\rangle
-\langle y(0),x^{*}(0)\rangle].
\end{aligned}
\end{equation*}
By applying It\^{o}'s formula for semi-martingales to
$\langle x^{*},y \rangle-\langle \psi,\varphi^{*}\rangle$, we have
\begin{equation*}
\begin{aligned}
&E[\langle G_{2}(\alpha(T))x^{*}(T)+\widehat{G}_{2}(\alpha(T))\widehat{x}^{*}(T),x^{*}(T)\rangle
-\langle y(0),x^{*}(0)\rangle]\\
=&E\bigg[\int_{0}^{T}\Big(-\langle Q_{2}x^{*},x^{*}\rangle
-\langle \widehat{Q}_{2}\widehat{x}^{*},x^{*}\rangle
+\langle u_{2}^{*},\mathbb{B}_{2}^{\top}y+\mathbb{D}_{2}^{\top}z
+\mathbb{F}_{2}\psi+\widehat{\mathbb{F}}_{2}\widehat{\psi}\rangle\Big)dt\bigg]\\
=&E\bigg[\int_{0}^{T}\Big(-\langle Q_{2}x^{*},x^{*}\rangle
-\langle \widehat{Q}_{2}\widehat{x}^{*},x^{*}\rangle
+\langle u_{2}^{*},\mathbf{B}_{2}^{\top}Y
+\mathbf{D}_{2}^{\top}Z+\mathbf{F}_{2}X
+\widehat{\mathbf{F}}_{2}\widehat{X}\rangle\Big)dt\bigg],
\end{aligned}
\end{equation*}
which implies that (noting (\ref{leader optimal condition}))
\begin{equation*}
\begin{aligned}
J_{2}(u_{1}^{*}(\cdot),u_{2}^{*}(\cdot))
=\langle Y(0),X(0)\rangle=\langle \widetilde{P}_{2}(0,i)X(0),X(0)\rangle
=\langle \widetilde{P}_{2}^{(11)}(0,i)x_{0},x_{0}\rangle.
\end{aligned}
\end{equation*}
On the other hand, note that $u_{2}^{*}$ defined by (\ref{leader optimal control})
for the leader is non-anticipating, thereby $u_{1}^{*}$ defined by
(\ref{follower optimal control}) for the follower can be also represented
in a non-anticipating way, i.e.,
\begin{equation}\label{follower non anticipating}
\begin{aligned}
u_{1}^{*}=&-\widetilde{N}_{1}^{-1}\Big[S_{1}x+\widehat{S}_{1}\widehat{x}+\Phi\Big]\\
=&-\widetilde{N}_{1}^{-1}\Big[(
\begin{array}{cc}
S_{1} & 0 \\
\end{array}
)X+(
\begin{array}{cc}
\widehat{S}_{1} & 0 \\
\end{array}
)\widehat{X}
+(
\begin{array}{cc}
0 & B_{1}^{\top} \\
\end{array}
)Y+(
\begin{array}{cc}
0 & D_{1}^{\top} \\
\end{array}
)Z+D_{1}^{\top}P_{1}D_{2}u_{2}^{*}\Big]\\
=&-\widetilde{N}_{1}^{-1}\Big[(
\begin{array}{cc}
S_{1} & 0 \\
\end{array}
)+(
\begin{array}{cc}
0 & B_{1}^{\top} \\
\end{array}
)P_{2}+(
\begin{array}{cc}
0 & D_{1}^{\top} \\
\end{array}
)(I-P_{2}\mathbf{D}_{1})^{-1}P_{2}\mathbf{J}_{2}\\
&-(
\begin{array}{cc}
0 & D_{1}^{\top} \\
\end{array}
)(I-P_{2}\mathbf{D}_{1})^{-1}P_{2}\mathbf{D}_{2}\widetilde{N}_{2}^{-1}\mathbf{S}_{2}
-D_{1}^{\top}P_{1}D_{2}\widetilde{N}_{2}^{-1}\mathbf{S}_{2}\Big]X\\
&-\widetilde{N}_{1}^{-1}\Big[(
\begin{array}{cc}
\widehat{S}_{1} & 0 \\
\end{array}
)+(
\begin{array}{cc}
0 & B_{1}^{\top} \\
\end{array}
)\widehat{P}_{2}+(
\begin{array}{cc}
0 & D_{1}^{\top} \\
\end{array}
)(I-P_{2}\mathbf{D}_{1})^{-1}P_{2}\widehat{\mathbf{J}}_{2}\\
&-(
\begin{array}{cc}
0 & D_{1}^{\top} \\
\end{array}
)(I-P_{2}\mathbf{D}_{1})^{-1}P_{2}\mathbf{D}_{2}\widetilde{N}_{2}^{-1}\widehat{\mathbf{S}}_{2}
-D_{1}^{\top}P_{1}D_{2}\widetilde{N}_{2}^{-1}\widehat{\mathbf{S}}_{2}\Big]\widehat{X}.
\end{aligned}
\end{equation}
The proof is completed.
\end{proof}
\begin{remark}
Up to now, we have completely solved our LQ leader-follower stochastic differential game for
mean-field switching diffusion. It turns out that the game admits an open-loop Stackelberg
equilibrium $(u_{1}^{*},u_{2}^{*})$ with a non-anticipating state feedback representation
(\ref{follower non anticipating}) and (\ref{leader optimal control}), respectively.
\end{remark}
Finally, we provide a numerical example to illustrate the effectiveness of our
theoretical results. Note that the optimal controls (\ref{follower non anticipating})
for the follower and (\ref{leader optimal control}) for the leader as well as
the value of the game (\ref{leader optimal cost}) depend only on the solutions
$P_{1}$, $\widetilde{P}_{1}$, $P_{2}$, $\widetilde{P}_{2}$ to Riccati equations
(\ref{follower 1}), (\ref{follower 5}), (\ref{leader 1}), (\ref{leader 4}), respectively.
So, in order to implement our control policies in practice, the \emph{whole} task for us
is to compute $P_{1}$, $\widetilde{P}_{1}$, $P_{2}$, $\widetilde{P}_{2}$.
\begin{example}
Let $n=m_{1}=m_{2}=1$ and $T=1$. Consider the following state equation:
\begin{equation*}
\left\{
\begin{aligned}
dX(t)=&[B_{1}(\alpha(t))u_{1}(t)+B_{2}u_{2}(t)]dt+CX(t)dW(t),\\
X(0)=&x_{0},
\end{aligned}
\right.
\end{equation*}
where $\alpha(\cdot)$ is a two-state Markov chain taking values in $\mathcal{M}=\{1,2\}$
with generator
\begin{equation*}
\begin{aligned}
\left[
  \begin{array}{cc}
    -1 & 1 \\
    1 & -1 \\
  \end{array}
\right],
\end{aligned}
\end{equation*}
and $B_{1}(1)=2$, $B_{1}(2)=1$, $B_{2}=1$, $C=0.5$.

The cost functionals for the follower and the leader are given by
\begin{equation*}
\begin{aligned}
J_{k}(u_{1}(\cdot),u_{2}(\cdot))=\frac{1}{2}
E\bigg[\int_{0}^{1}N_{k}u_{k}^{2}(t)dt+G_{k}X^{2}(1)
+\widehat{G}_{k}(E[X(1)|\mathcal{F}_{1}^{\alpha}])^{2}\bigg],
\end{aligned}
\end{equation*}
where $N_{k}=1$, $G_{k}=1$, $\widehat{G}_{k}=0.5$, $k=1,2$, respectively.
Note that in this example, to exhibit the effect of regime switching more clearly,
we only let $B_{1}$ vary depending on the Markov chain and keep all the other parameters
fixed as constants.

Then, $P_{1}(t,i)$, $\widetilde{P}_{1}(t,i)$, $P_{2}^{(11)}(t,i)$, $\widetilde{P}_{2}^{(11)}(t,i)$,
$i\in\{1,2\}$, on $[0,1]$ are computed and plotted in Figures \ref{FR} and \ref{LR}, respectively.
It is mentioned that the other elements of the matrix-valued functions $P_{2}(t,i)$ and
$\widetilde{P}_{2}(t,i)$, $i\in\{1,2\}$, are not plotted for simplicity.
\begin{figure}
\centering
\includegraphics[width=5.5in]{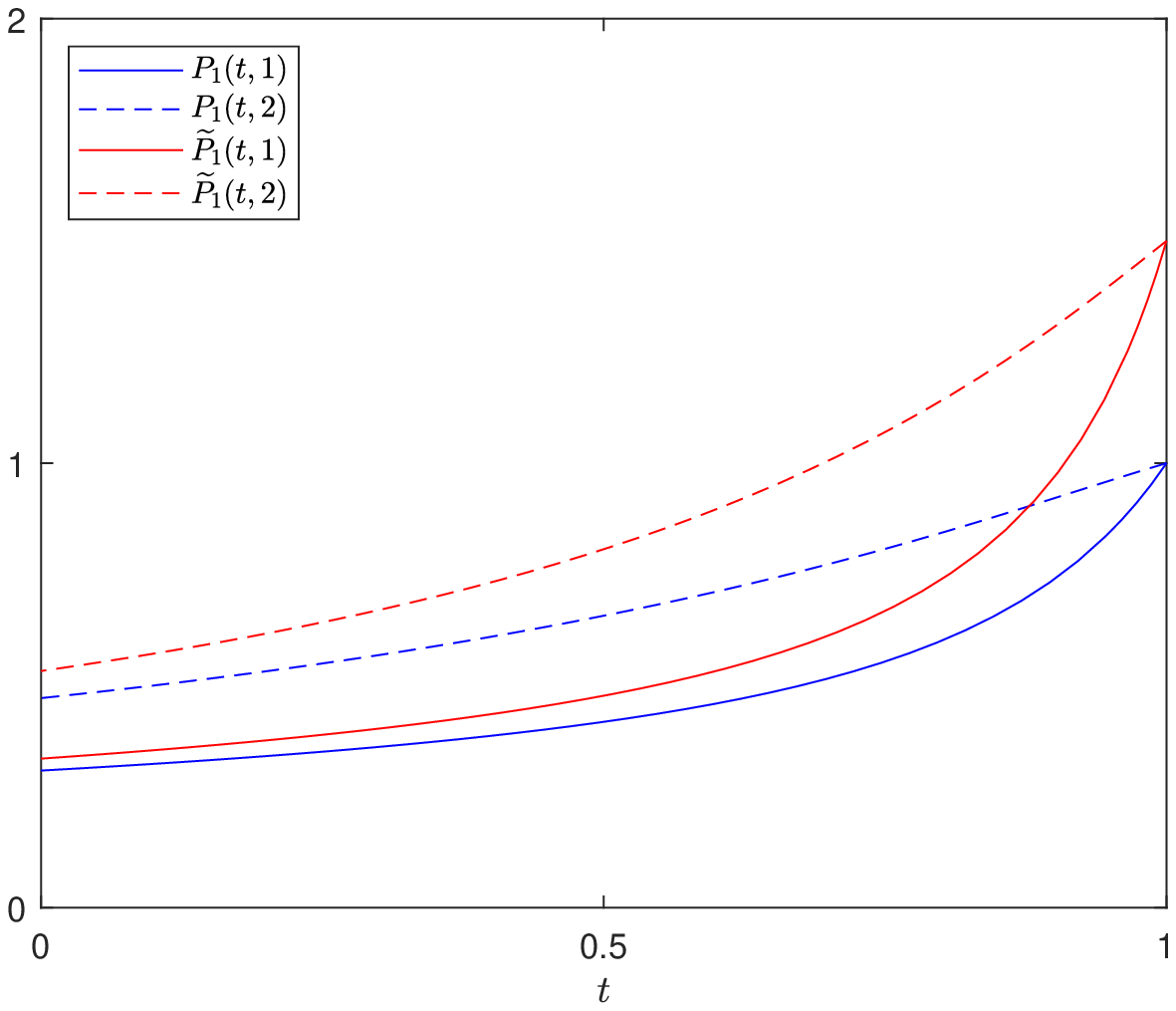}
\caption{Riccati equations for the follower}\label{FR}
\includegraphics[width=5.5in]{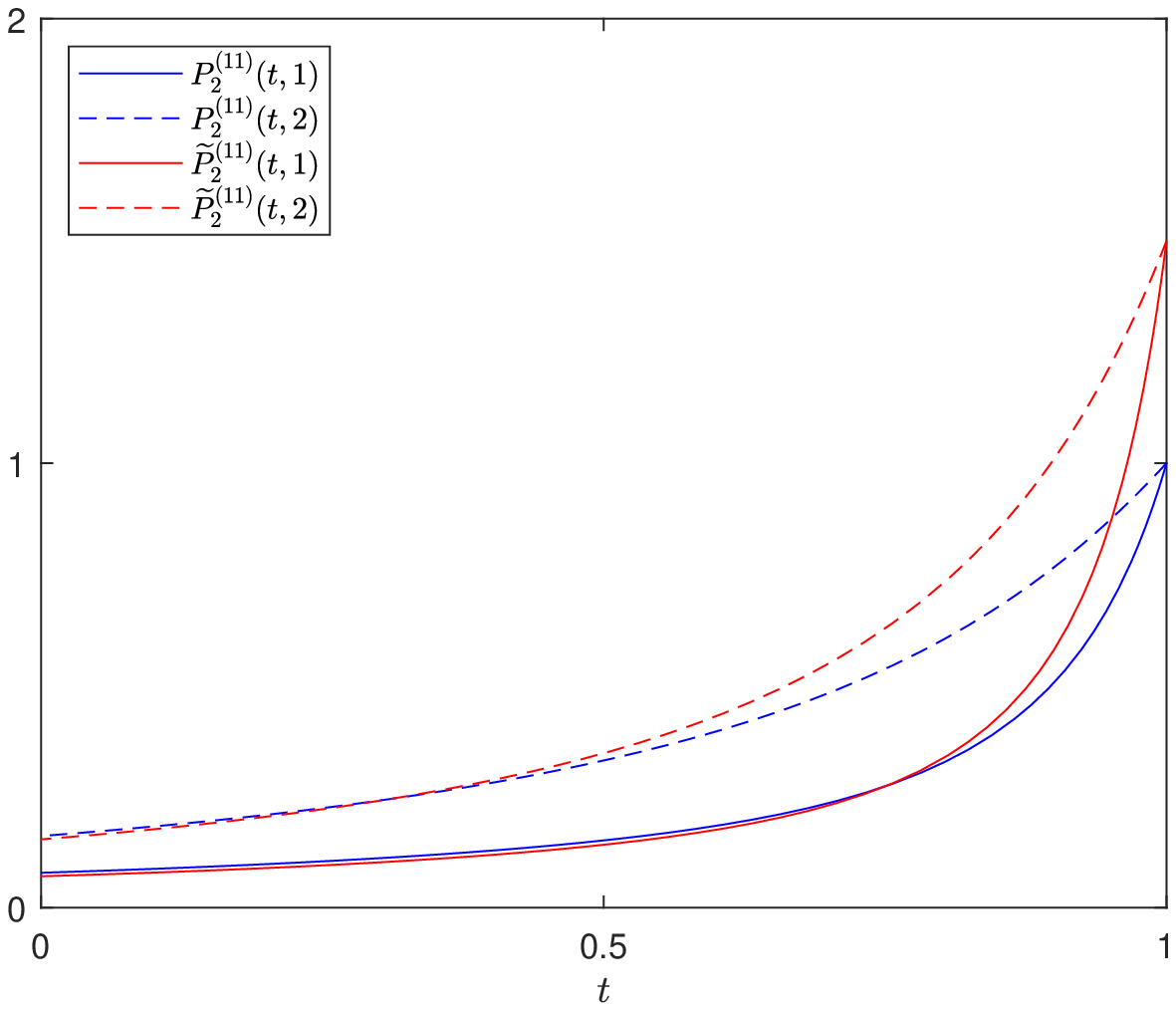}
\caption{Riccati equations for the leader}\label{LR}
\end{figure}
\end{example}

\section{Concluding remarks}\label{conclusion}

In this paper, we studied an LQ leader-follower stochastic differential game
with regime switching and mean-field interactions. Conditional mean-field terms
are included due to the presence of a Markov chain (just like a \emph{common noise}).
Some new-type Riccati equations are introduced for the first time in the literature.
The open-loop Stackelberg equilibrium and its non-anticipating state feedback
representation are obtained. There are several interesting problems that deserve
further investigation, in particular, the existence and uniqueness results of
the Riccati equations (\ref{leader 1}) and (\ref{leader 4}).

\end{document}